\documentclass[11pt]{article}
\usepackage[T1]{fontenc}
\usepackage{latexsym,amssymb,amsmath,amsfonts,amsthm}
\usepackage{mathtools}
\usepackage{dsfont}
\usepackage{graphicx}
\usepackage{subfigure}
\usepackage{placeins}
\usepackage{color}
\usepackage{indentfirst}
\usepackage{enumerate}
\definecolor{hw}{rgb}{0,0,0}
\definecolor{hw1}{rgb}{0,0,0}
\definecolor{hw2}{rgb}{0,0,0}
\definecolor{hw3}{rgb}{1,0,0}
\definecolor{hw4}{rgb}{1,0,0}
\numberwithin{equation}{section}
\newcommand{\R}{{\mathbb R}}

\newcommand{\N}{{\mathbb N}}
\newcommand{\C}{{\mathbb C}}
\newcommand{\s}{{\mathbb S}}

\newcommand{\be}{\begin{eqnarray}}
\newcommand{\ben}{\begin{eqnarray*}}
\newcommand{\en}{\end{eqnarray}}
\newcommand{\enn}{\end{eqnarray*}}

\newtheorem{theorem}{Theorem}[section]
\newtheorem{lemma}[theorem]{Lemma}

\newtheorem{definition}[theorem]{Definition}
\newtheorem{remark}[theorem]{Remark}
\newtheorem{re}[theorem]{Remark}

\definecolor{rot}{rgb}{1.000,0.000,0.000}
\definecolor{blue}{rgb}{0.000,0.000,1.000}
\newcommand{\tcb}{\textcolor{black}}
\newcommand{\tcr}{\textcolor{black}}

\begin{document}
\renewcommand{\theequation}{\arabic{section}.\arabic{equation}}
\begin{titlepage}
\title{\bf \tcb{Piecewise-analytic} interfaces with weakly singular points of arbitrary order always scatter}


\author{
Long Li\thanks{NCMIS and Academy of Mathematics and Systems Science, Chinese Academy of Sciences,
Beijing 100190, China ({\tt longli@amss.ac.cn}).
}
\and Guanghui Hu\thanks{Corresponding author: School of Mathematical Sciences and LPMC, Nankai University, Tianjin 300071, P.R. China ({\tt
ghhu@nankai.edu.cn}).}
\and Jiansheng Yang \thanks{LMAM, School of Mathematical Sciences, Peking University, Beijing 100871,China ({\tt jsyang@pku.edu.cn}).
}
 }
\date{}
\end{titlepage}
\maketitle
\vspace{.2in}

\begin{abstract}
It is proved that an inhomogeneous medium whose boundary contains a weakly singular point of arbitrary order scatters every incoming wave. Similarly, a compactly supported source term with weakly singular points on the boundary always radiates acoustic waves. These results imply the absence of non-scattering energies and non-radiating sources in a domain \tcb{whose boundary is piecewise analytic but not infinitely smooth}. Local uniqueness results with a single far-field pattern are obtained for inverse source and inverse medium scattering problems. \tcr{Our arguments provide a rather weak condition on  scattering interfaces and refractive index functions to guarantee the scattering phenomena that the scattered fields cannot vanish identically.}

\vspace{.2in} {\bf Keywords}: Non-scattering energy; non-radiating source; weakly singular points; inverse medium scattering; inverse source problem; uniqueness.
\end{abstract}

\section{Introduction and main results}
Let $D\subset \R^2$ be a bounded domain such that its exterior $D^e:=\R^2\backslash\overline{D}$ is connected. In this paper, the domain $D$ represents either the support of an acoustic source $s\in L^\infty(\R^2)$ or the support of the contrast function $1-q$ of an inhomogeneous medium. The source term $s$ is called non-radiating if it does not radiate wave fields at infinity. Analogously, if a penetrable obstacle $D$ scatters any incoming wave trivially at the wavenumber $\kappa>0$, then $\kappa$ is called a non-scattering energy. The study of non-scattering energies dates back to \cite{KS} in the case of a convex (planar) corner domain, where the notion of \emph{scattering support} for an inhomogeneous medium was explored.
 The existence of non-radiating sources and non-scattering energies may cause essential difficulties in detecting a target from far-field measurements. It is well known that non-scattering energies and non-radiating sources can be excluded if $\partial D$ contains a curvilinear polygonal/polyhedral corner or a circular conic corner, in other words, corners always scatter; see e.g.,
\cite{BLS, ElHu2015, EH2017, HSV, KS, PSV}. The aforementioned "visible" corners can be interpreted as \emph{strongly singular} points, since the first derivative of the function for parameterizing $\partial D$ is usually discontinuous at these points. In this paper,
the boundary $\partial D$ is supposed to be $C^1$-smooth and piecewise analytic with a finite number of \emph{weakly singular} points.

Throughout the paper we set $\N_0:=\N\cup\{0\}$ and $B_\epsilon(x):=\{y=(y_1,y_2)\in\R^2: |y-x|<\epsilon\}$. Write $B_\epsilon=B_\epsilon(O)$ where $O=(0,0)$ always denotes the origin. Set $\partial_j:=\partial/\partial x_j$ for $j=1,2$. Below we state the definition of a weakly singular point.

\begin{definition}\label{d0}
The point $O\in \partial D$ is called a weakly singular point of order $m\geq 2$ ($m\in \N$) if the subboundary $B_{\epsilon}(O)\cap \partial D$ for some $\epsilon>0$ can be parameterized by the polynomial $x_2= f(x_1), x_1\in(-\epsilon/2,\epsilon/2)$, where
\begin{equation}
\label{eq:6}
f(x_1)=\left\{\begin{array}{lll}
\sum_{n
\in \N_0} \frac{f^{+}_{n}}{n!}x_1^n, &&\quad -\epsilon/2<x_1\leq0, \\
\sum_{n
\in \N_0} \frac{f^{-}_{n}}{n!}x_1^n, &&\qquad 0\leq x_1<\epsilon/2. \end{array}\right.
\end{equation}
Here, the real-valued coefficients $\{f^{\pm}_{n}\}^{\infty}_{n=1}$ satisfy the relations
\[
f^{+}_{l}=f^{-}_{l}:=f_l, \quad\forall\; 0\leq l<m\quad
\mbox{and}\qquad f^{+}_{m}\ne f^{-}_{m},
\]
with $f_{l}=0$ for $l = 0,1$.
\end{definition}
We require $m\geq 2$ in the above definition, because
 a singular point of order one is exactly strongly singular in the sense that $f$ is continuous and the first derivative $f':=df/dx_1$ is discontinuous at $O$.
 Obviously, each planar corner point of a polygon is strongly singular. If $O$ is weakly singular, then $\partial D$ is piecewise analytic but cannot be  $C^\infty$-smooth at this point. The purpose of this paper is to prove that
 \begin{description}
 \item[(i)] An inhomogeneous medium with a weakly singular point of arbitrary order lying on the support $D$ of the contrast function scatters every incoming wave (Theorem \ref{TH}).
 \item[(ii)] A source term embedded in an inhomogeneous medium with a weakly singular point of arbitrary order lying on the support $D$ of the source function always radiates acoustic waves non-trivially (Theorem \ref{th1}).
 \item[(iii)] Local uniqueness results in recovering source terms and the shape of an inhomogeneous medium (Theorems \ref{th2} and \ref{TH2}).
 \end{description}
 The first/second assertion implies the absence of non-scattering energies/non-radiating sources, when the piecewise analytic boundary $\partial D$ contains at least one weakly singular point of arbitrary order. We formulate and remark our main results below.

\subsection{Radiating sources in an inhomogeneous medium}
Consider the radiating of a time-harmonic acoustic source  in an inhomogeneous background medium  in two dimensions. This can be modeled by the inhomogeneous Helmholtz equation
	\be \label{eq:1-cornscatter2018}
		 \Delta v(x) + \kappa^2 \mathfrak{n}(x) v(x) = s(x) \quad \mbox{in} \quad \R^2.
	\en	
In this paper, the potential (or refractive index) function $\mathfrak{n}$ of the inhomogeneous background medium is supposed to be real-analytic in $B_R$ and
	$\mathfrak{n}(x)\equiv1$ in  $|x|>R$ for some $R>0$. The number $\kappa>0$ represents the wavenumber of the homogeneous medium in $|x|>R$ and $s\in L^2(\R^2)$ is a source term compactly supported in $D\subset B_R$. Further, it is supposed that $s=S|_{\overline{D}}$ where $S$ is a real-analytic function defined in a neighborhood of $D$. Since $v$ is outgoing at the infinity, it satisfies the Sommerfeld radiation condition
\begin{equation}\label{eq:radiation}
\lim_{|x|\rightarrow \infty} \sqrt{r}\,\left\{ \frac{\partial v}{\partial r}-i\kappa v \right\}=0,\quad r=|x|,
\end{equation}
uniformly in all directions $\hat x:=x/|x|\in\s:=\{x\in\mathbb{R}^2: |x|=1\}$.
In particular, the Sommerfeld radiation condition (\ref{eq:radiation}) leads to  the asymptotic expansion
\ben
v(x)=\frac{e^{i\kappa r}}{\sqrt{r}}\; v^\infty(\hat x)+\mathcal{O}\left(\frac{1}{r^{3/2}}\right),\quad r\rightarrow+\infty.
\enn
  The function $v^\infty(\hat x)$ is an analytic function defined on $\s$ and is usually referred to as the \emph{far-field pattern} or the \emph{scattering amplitude}.  The vector $\hat{x}\in\s$ is called the observation direction of the far-field pattern.
  Using the variational approach, one can readily prove that the system \eqref{eq:1-cornscatter2018} admits a unique solution in $H^2_{loc}(\R^2)$; see \cite[Chapter 5]{CK} or \cite[Chapter 5]{Cakoni}. Since the far-field pattern encodes information on the source,  we are interested in the inverse problem of recovering the source support $\partial D$ and/or the source term $s(x)$ from the far-field pattern over all observation directions at a fixed frequency.

   The source term $s(x)$ is called  \emph{non-radiating} if $v^\infty$ vanishes identically. For example, setting $s:=(\Delta+\kappa^2\mathfrak{n}(x))\varphi$  for some $\varphi\in C_0^\infty(B_R)$, it is easy to observe that the unique radiating solution to (\ref{eq:1-cornscatter2018}) is exactly $\varphi$, which has the vanishing far-field pattern. Hence, in general a single far-field pattern cannot uniquely determine a source function (even its support), due to the existence of non-radiating sources. In the following two theorems, we shall characterize a class of radiating sources and extract partial or entire information of an analytical source term at a weakly singular point.

  \begin{theorem}[Characterization of radiating sources]\label{th1}
  If $O\in\partial D$ is a weakly singular point such that $|s(O)|+|\nabla s(O)|>0$, then $v^\infty$ cannot vanish identically. Further, the wave field $v$ cannot be analytically continued from $B_R\backslash\overline{D}$ to $B_\epsilon(O)$ for any $\epsilon>0$.
  \end{theorem}

  \begin{theorem}[Determination of source term]\label{th2} Assume that $D$ and $\mathfrak{n}$ are both known in advance and that $O\in\partial D$ is a weakly singular point. Then
  \begin{itemize}
  \item[(i)] The far-field pattern $v^\infty$ uniquely determines the values of $s$ and $\nabla s$ at $O$.
  \item[(ii)] Suppose additionally that the source term $s(x)$ satisfies the elliptic equation
\be \label{N:15}
\Delta s(x) + A(x)\cdot \nabla s(x) + b(x) s(x)=0 \ \text{on} \ \overline{D},
\en
where $A(x)\in (L^\infty(B_R))^2$ and $b(x)\in L^\infty(B_R)$ are given functions that are real-analytic around $O$.
Then $s(x)$ can be uniquely determined by $v^\infty$.
\end{itemize}
\end{theorem}

The admissible source functions satisfying (\ref{N:15}) process the property that the lowest order Taylor expansion at $O$ is harmonic (see \cite{HL}), that is, for some $N\in \N_0$,
\ben \label{eq:hform1-cornscatter2018}
		s(x) = r^N\left(A \cos (N \theta) + B \sin (N\theta)\right) + \mathcal O(r^{N+1}), \quad |x| \to 0,\, x \in B_\epsilon(O).
		\enn
Theorems \ref{th1} and \ref{th2} have generalized the results of \cite{HL, Bl2018} for planar corners to the case of arbitrarily weakly singular points (in the sense of Definition \ref{d0}), under the analytical assumptions imposed on $n$ and $s$. Without these a priori assumptions,
one can prove uniqueness by using multi-frequency near/far field data; we refer to \cite{AM, EV} for the uniqueness proof in a homogeneous background medium and to \cite{BLT,CIL} for increasing stability estimates in terms of the bandwidth of frequencies.

\subsection{Absence of non-scattering energies}
Let $D\subset \R^2$ be a bounded penetrable
scatterer embedded in a homogeneous isotropic background medium. The acoustic properties of $D$ can be characterized by the refractive index function $q\in L^2_{loc}(\R^2)$ such that $q\equiv 1$ in $D^e$ after some normalization. Hence
the contrast function $1-q$ is compactly supported in $D$. Assume that a time-harmonic non-vanishing incoming wave $u^{in}$ is incident onto $D$, which is governed by the Helmholtz equation
$(\Delta +\kappa^2) u^{in}=0$ at least in a neighborhood of $D$. For instance, $u^{in}$ is allowed to be a plan wave, a Herglotz wave function or a cylindrical wave  emitting from a source position located in $\R^2\backslash\overline{D}$.
  The wave propagation of the total field $u=u^{in}+u^{sc}$ is then modeled by the Helmholtz equation
\ben
\Delta u+\kappa^2 q\, u=0\quad\mbox{in}\quad \R^2.
\enn
 At the infinity, the perturbed scattered field
 $u^{sc}$ is supposed to fulfill the Sommerfeld radiation condition (\ref{eq:radiation}).
The unique solvability of the above medium scattering problem in $H^2_{loc}(\R^2)$ is well known (see e.g., \cite[Chapter 8]{CK}). We suppose that $q$ is real-analytic on $\overline{D}$, that is, there exists a real-analytic function $Q$ defined in a neighborhood of $D$ such that $Q|_{\overline{D}}=q$. Further, we suppose that $|q(O)-1|+|\partial_1q(O)|>0$ for each weakly singular point $O$ lying on $\partial D$, because of the medium discontinuity.
For instance, $q(x)=q_0+q_1 x_1+q_2x_2$ on $\overline{D}$ where $q_0,q_1, q_2\in \R$ satisfying $|q_0-1|+|q_1|>0$.
This covers at least the piece-wise constant case that $q|_{\overline{D}}\equiv q_0\neq 1$.
We shall prove that
\begin{theorem}[Weakly singular points always scatter]\label{TH}
The penetrable scatterer $D\subset \R^2$ scatters every incoming wave, if $\partial D$ contains at least one weakly singular point $O$ (see Def. \ref{d0}).  Further, $u$ cannot be analytically continued from $\R^2\backslash\overline{D}$ to $B_\epsilon(O)$ for any $\epsilon>0$.
\end{theorem}
As a by-product of the proof of Theorem \ref{TH},
 we get a local uniqueness result to the shape identification with a single incoming wave.
 \begin{theorem}\label{TH2} Let $D_j$ ($j=1,2$) be two penetrable scatterers in $\R^2$ with the analytical potential functions $q_j$, respectively.  If $\partial D_2$ differs from $\partial D_1$ in the presence of a weakly singular point lying on the boundary of the unbounded component of $\R^2\backslash\overline{(D_1\cup D_2)}$, then the far-field patterns corresponding to $(D_j,q_j)$ incited by any non-vanishing incoming wave cannot coincide.
\end{theorem}
 Here we mention the connection of Theorems \ref{TH} and \ref{TH2} with a cloaking device. \tcb{ The latter always leads to vanishing observation data and is closely related to uniqueness in inverse scattering.}
  It follows from Theorem \ref{TH} that a cloaking device cannot be designed by homogeneous and isotropic medium with a weakly singular point lying the boundary surface. There are essential difficulties in our attempt to prove the global uniqueness with a single far-field pattern. To the best of our knowledge, such kind of global uniqueness for shape identification remains open for a long time, \tcb{ since Schiffer's first result using infinitely many plane waves in 1967 (see \cite{Lax}). Theorem \ref{TH2} has partly answered this open question.}

 The second assertions in Theorems \ref{th1} and \ref{TH} imply that, the wave field must be "singular" (that is, non-analytic) at the weakly singular points.
 Excluding the possibility of analytical extension turns out to be helpful in designing non-iterative inversion algorithms for locating planar corners; see e.g., the enclosure method \cite{Ik1999,Ik2000}, the
 one-wave version of range test approach  \cite{KPS, KS} and no-response test method \cite{LNPW,P2003,PS2010} as well as the data-driven scheme recently prosed in \cite{EH2019,HL}.   Most of these inversion schemes can be interpreted as domain-defined sampling methods (or analytic continuation tests, see \cite[Chapter 15]{NP2013} for detailed discussions), in comparison with the pointwise-defined sampling approaches such as Linear Sampling Method \cite{Cakoni}, Factorization Method \cite{Kirsch08} and Point Source Method \cite{Potthast} etc. Combining the ideas of \cite{KPS, KS, EH2019} with our results, one may conclude that arbitrarily weakly singular points lying on the convex hull of $D$ can be numerically reconstructed from the data of a single far-field pattern.

  In our previous work \cite{LHY}, the analogue results to Theorems \ref{TH2} and \ref{TH} were verified in a piece-wise constant medium where the locally parameterized boundary function $f$ takes the special form (cf. (\ref{eq:6}))
  \begin{align*}
f(x_1)=\left\{\begin{array}{lll}
f_j^+\, x_1^{j}, &&\quad \epsilon/2>x_1\geq 0,\\
f_n^-\, x_1^{n},&& -\epsilon/2<x_1\leq 0,
\end{array}\right.\quad j, n\in \N_0,\quad f_j^+, f_n^-\in \R,
\end{align*}
 with the conditions
 \[
 j,n\geq 2, \quad (f_j^+, j)\neq (f_n^-, n),\quad (f_j^+)^2+(f_n^-)^2\neq 0.
 \]
 Obviously, the weakly singular points and potential functions considered in this paper  are more general than those in \cite{LHY}. In fact, the mathematical techniques and algebraic calculations in the present paper are more subtle and intricate than \cite{LHY}.

For strongly singular corners \cite{EH2017,HSV}, the smoothness of the potential function can be even weakened  to be H\"older continuous with a lower contrast to the background medium (that is, $1-q$ is $C^{k}$-smooth at $O$ for some $k\geq 2$). Using additionally involved arguments, our approach can also handle the lower contrast case. However, we only consider the higher contrast medium fulfilling the condition $|q(O)-1|+|\partial_1 q(O)|>0$, since the emphasis of this paper is placed upon treating interfaces with weakly singular points of arbitrarily order $m\geq 2$.

 \vspace*{12pt}

\tcr{In the recent article \cite{Cav}, it is revealed that in $\R^n$ ($n=2,3$), if the Lipschitz boundary $\partial D$ processes a non-analytic point $O$ and if $q$ is analytic in a neighborhood of $O$, then the object $D$ scatters the incoming wave $u^{i}$, provided that $(q(O)-1)u^i(O) \ne 0$. The authors have also proved the same result under a weaker regularity assumption on $q$ ($C^{m,\alpha}$-smooth) but with a strong regularity assumption on $\partial D$ ($C^{m+1,\alpha}$-smooth) near $x_0$.
The approach in \cite{Cav} relies on free boundary methods.  Such a method has been also employed in \cite{SaSh} to show that penetrable scatterers with real-analytic boundaries admit an incident wave such that it doesn't scatter. Under the non-vanishing condition of the incident wave, a necessary condition on a boundary point was derived for the non-scattering phenomena.} However, nothing is known for general smooth domains without the non-vanishing condition on the incident wave and the contrast $q-1$. Our results illustrate that piecewise analytic penetrable scatterers (at least $C^1$-smooth but not $C^{\infty}$-smooth) scatter any incoming waves, under the analytic assumption on $q$ together with the medium discontinuity assumption $|q-1|+|\partial_1 q|>0$ at the weakly singular point. Moreover, we establish a local uniqueness result for shape reconstruction with a single incoming wave.

 The remaining part is organized as follows. Our main efforts will be spent on an analytical approach to the proof of Theorem \ref{TH}
in Section \ref{sec:WSS}. This also yields the local uniqueness result of Theorem \ref{TH2}. In Section \ref{sec:Source}, we shall adapt this approach to prove Theorems \ref{th1} and \ref{th2}. The proofs of some important Lemmata will be given in the Appendix.

\section{Weakly singular points always scatter}\label{sec:WSS}

\subsection{Proofs of Theorems \ref{TH} and \ref{TH2}.}
This section is devoted to the proofs of Theorems \ref{TH} and \ref{TH2} when the penetrable scatterer $\partial D$ contains a weakly singular point on the boundary.  For this aim, we first generalize the Cauchy-Kovalevski theorem for the Helmholtz equation to a piecewise analytic domain.

\begin{lemma}\label{nle}
Let $D$ be a domain in $\R^2$ and suppose that the boundary $\partial D$ in an $ \epsilon$-neighborhood of $O\in \partial D$ can be represented by $\Gamma=\{(x_1,f(x_1)): x_1\in(-\epsilon/2,  \epsilon/2)\}$, where $f$ is given by (\ref{eq:6}). Let $\widetilde u$ satisfies
\ben
&\Delta  \widetilde u + \widetilde q(x) \widetilde u = 0 \quad\mbox{in}\quad D \cap B_{ \epsilon},\\
&\widetilde u = \widetilde g_0, \; \partial_\nu \widetilde u = \widetilde  g_1   \quad\mbox{on}\quad \Gamma,
\enn
where $\widetilde q = \widetilde Q|_{\overline D}, \widetilde g_0 = \widetilde G_0|_\Gamma, g_1 = \widetilde G_1|_\Gamma$ with $\widetilde Q, \widetilde G_0$ and $\widetilde G_1$ being analytic in $B_\epsilon$.  Then, there exits $ \epsilon _1 \in (0, \epsilon)$ such that $\widetilde u$ can be extended analytically from $\overline{D}\cap B_{ \epsilon_1}$ to $B_{ \epsilon_1}$ and the extended function $\widetilde u_1$ satisfies that
\begin{align*}
\Delta \widetilde u_1 + \widetilde Q(x) \widetilde u_1 = 0  \quad\mbox{in}\quad  B_{ \epsilon_1}.
\end{align*}
\end{lemma}
\begin{proof}
Since $\Gamma$ is piecewise analytic, we may extend one of its analytic components to a well-defined  analytic function over the interval $(-\epsilon_1/2, \epsilon_1/2)$ for some $\epsilon_1\in(0, \epsilon)$ sufficiently small.  Define $\Gamma_1: = \{(x_1, f^+(x_1)): x_1\in(- \epsilon_1/2, \epsilon_1/2)\}$ with
\begin{align*}
f^+(x_1) = \sum_{n
\in \N_0} \frac{f^{+}_{n}}{n!}x_1^n, &\quad -\epsilon_1/2<x_1\leq \epsilon_1/2.
\end{align*}
Here, $\{f^{+}_{n}\}_{n\in \N_0}$ is given by (\ref{eq:6}).
By the Cauchy-Kovalevski theorem (see e.g. \cite[Chapter 3.3]{John}), the Cauchy problem of the Helmholtz equation
\begin{align*}
&\Delta \widetilde u_1 + \widetilde Q(x) \widetilde u_1 = 0  \quad \qquad\quad\mbox{in}\quad  B_{ \epsilon_1}, \\
&\widetilde u_1 = \widetilde G_0|_{\Gamma_1}, \; \partial_\nu \widetilde u_1 = \widetilde  G_1|_{\Gamma_1}   \quad\mbox{on}\quad \Gamma_1
\end{align*}
admits only a unique analytic solution $\widetilde u_1$. Note that
\begin{align*}
&\Delta (\widetilde u_1- \widetilde u) + \widetilde q(x) (\widetilde u_1- \widetilde u) = 0  \quad\mbox{in}\quad  D \cap  B_{ \epsilon_1},\\
&  \widetilde u_1 = \widetilde u, \quad \partial_\nu \widetilde u_1 = \partial_\nu \widetilde u,  \quad  \qquad \quad\mbox{on}\quad \Gamma^+ \cap B_{ \epsilon_1},
\end{align*}
where $\Gamma^+ = \{(x_1,f(x_1))\in \Gamma: 0<x_1<\epsilon/2\}$. Then,
by Holmgren's theorem, it can be deduced that $\widetilde u_1 (x) = \widetilde u (x)$ for $x \in  D \cap B_{ \epsilon_1}$.  Therefore, $\widetilde u_1$ can be considered as an analytical extension of $\widetilde u$ from $\overline{D}\cap B_{ \epsilon_1}$ to $B_{ \epsilon_1}$. The proof of this lemma is thus completed.
\end{proof}

Since the Helmholtz equation remains invariant by coordinate translation and rotation, in the following analysis, we can suppose without loss of generality that the weakly singular point always coincides with the origin. Assume that the boundary $\partial D$ in an $\epsilon$-neighborhood of $O$ can be represented by $\Gamma=\{(x_1,f(x_1)): x_1\in(-\epsilon/2,\epsilon/2)\}$, where
the function $f$ is given by (\ref{eq:6}). Assuming that $u^{sc}$ vanishes in $D^e$, we shall derive a contradiction.
Across the interface $\partial D$, we have the continuity of the total field and its normal derivative, 
\be\label{TE1}
u^+=u^-,\quad \partial_\nu u^+=\partial_\nu u^-\quad\mbox{on}\;\partial D.
\en
Here the superscripts $(\cdot)^\pm$ stand for the limits taken from outside and inside, respectively, and $\nu\in \s$ is the unit normal on $\partial D$ pointing into $D^e$.
Since $u^{sc}=0$ in $D^e$,  the Cauchy data of $u$ on $\Gamma$ coincide with those of $u^{in}$, which are real-analytic since the incoming wave fulfills the Helmholtz equation near $D$.
Observing the fact that $q=Q|_{\overline{D}}$ is analytic on $\overline{D}$ and that $\Gamma$ is piecewise analytic, applying the Lemma \ref{nle}, one may analytically extend $u$ from
 $\overline{D}\cap B_\epsilon$ to a small neighborhood of $O$ in
 $D^e\cap B_\epsilon$.  For notational convenience, we still denote the extended domain by $B_\epsilon$ and the extended function by $u$,  satisfying the Helmholtz equation
\ben
\Delta u+\kappa^2 Q(x) u=0\quad\mbox{in}\quad B_\epsilon.
\enn
In the subsequent sections we take $\epsilon=1$ for simplicity (see Figure \ref{fig:1}).
Using the transmission condition (\ref{TE1}) together with $u^{sc}\equiv 0$ in $D^e$, we deduce that
\be\label{transmission}\left\{\begin{array}{lll}
\Delta w_j+q_j(x)w_j=0, &\mbox{ in}\quad B_1,\quad j=1,2,\\
w_1=w_2,\quad\partial_\nu w_1= \partial_\nu w_2&\mbox{ on}\quad \Gamma,
\end{array}\right.\en
where
\be\label{setting}
w_1=u^{in},\quad w_2=u,\qquad q_1(x)\equiv \kappa^2,\quad q_2(x)=\kappa^2 Q(x).
\en
\begin{figure}[htbp]
			\centering
			\includegraphics[width = 3.8in]{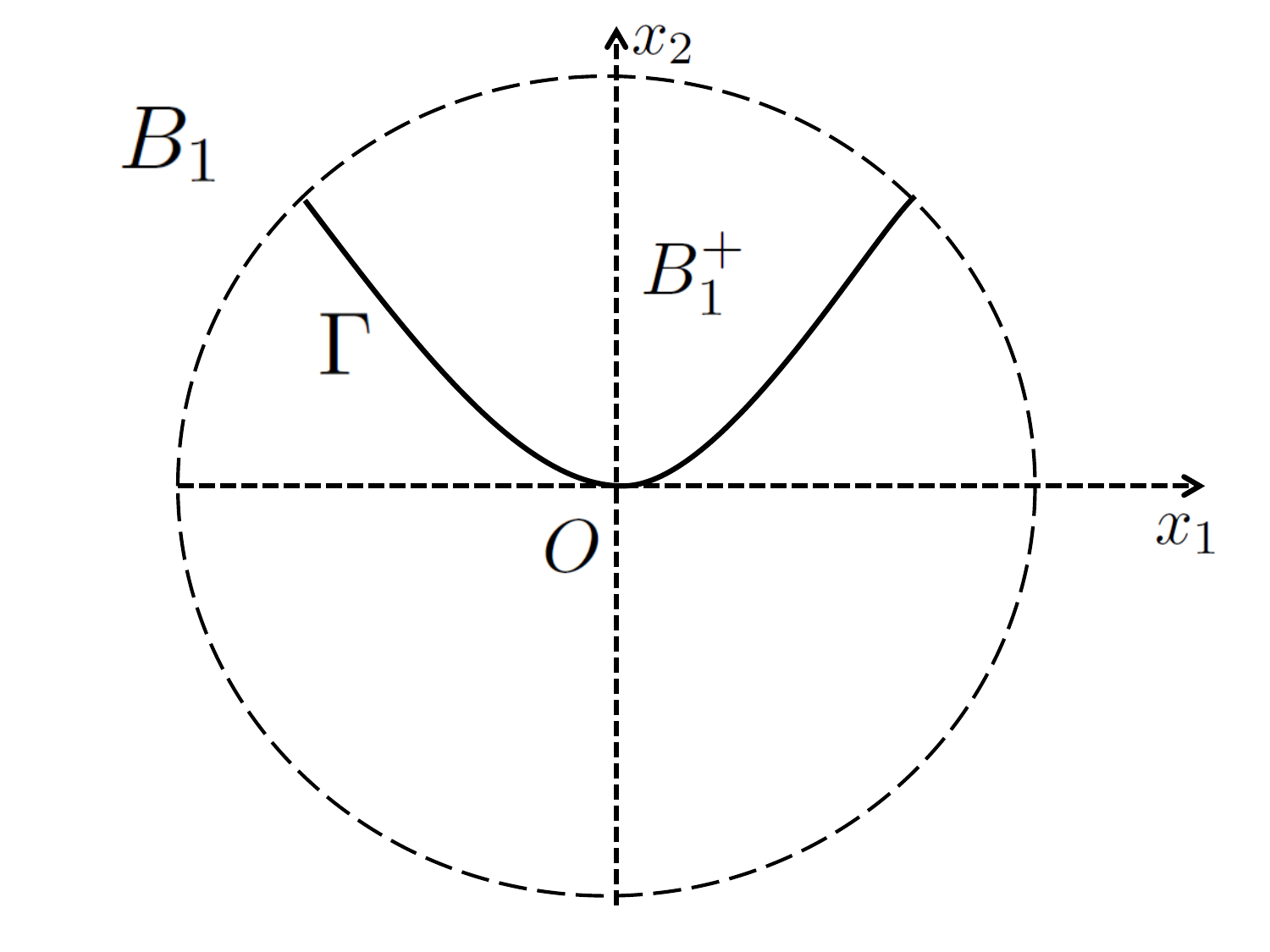}
			\caption{Illustration of the piecewise analytic interface $\Gamma\subset \R^2$ which contains a weakly singular point at $O=(0,0)$.}
			\label{fig:1}
		\end{figure}
We shall prove
\begin{lemma} {\label{D.1}}
Let $q_1$ and $q_2$ be  real-analytic functions defined in $B_1$.
  Suppose that $w_j\in H^2(B_1)$ ($j=1,2$) are solutions to (\ref{transmission}) and $O\in\Gamma$ is a weakly singular point with the local parametrization of the form (\ref{eq:6}). If
  \be\label{q}
  |(q_1-q_2)(O)|+|\partial_1(q_1-q_2)(O)|> 0
  \en
   then  $w_1=w_2\equiv0$ in $B_1$.
  \end{lemma}
\tcb{Lemma \ref{D.1} implies that the Cauchy data of two Helmhotz equations cannot coincide on a piecewise analytic curve with a weakly singular point, if the involved analytical potentials are not identical at this point. The result of Lemma \ref{D.1} is not valid if $q_1\equiv q_2$ near $O$. It also implies that, solutions to the Helmholtz equation $\Delta w_1+q_1(x) w_1=0$ in $B_1^+:=\{x\in B_1: x_2>f(x_1)\}$ (see Figure \ref{fig:1})
cannot be analytically continued into $B_1$, if the Cauchy data of $w_1$ coincide with those of $w_2$ on $\Gamma$. Hence, Lemma \ref{D.1} gives a sufficient condition of the boundary under which solutions to the Helmholtz equation admits no analytical extension. This is in contrast with the Cauchy-Kowalevski theorem (see e.g., \cite[Chapter 3.3]{John}) which guarantees a locally analytical extension of analytic solutions if both the Cauchy data and boundary surface are analytic. It seems that this local property of the Helmholtz equation also extends to other elliptic equations such as the Lam\'e system.
}

Based on Lemma \ref{D.1}, we can readily prove   Theorems \ref{TH} and \ref{TH2}.

{\bf Proof of Theorems \ref{TH} and \ref{TH2}.}
Let $O\in \partial D$ be a weakly singular point.
  We first note that the jump condition (\ref{q}) applies to the potentials given in (\ref{setting}), since $|q(O)-1|+|\partial_1q(O)|> 0$.
  By Lemma \ref{D.1} and the unique continuation, $u^{in}$ vanishes identically in $\R^2$ if $u^{sc}$ vanishes identically, which contradicts our assumption.
  Hence, an analytical potential with a weakly singular point lying on the boundary of the contrast function's support always scatters.
  This proves the first assertion of Theorem \ref{TH}. If the total field $u$ can be analytically continued from $\R^2\backslash\overline{D}$ to $B_\epsilon(O)$ for some $\epsilon>0$, one can get the same system as (\ref{transmission}), where $w_1$ is now replaced by the extension of $u$ satisfying  the Helmholtz equation with the wavenumber $\kappa^2$. By Lemma \ref{D.1} we get $u\equiv 0$, implying that $u^{sc}$ can be extended to an entire radiation solution to the Helmholtz equation. Hence we obtain the vanishing of $u^{sc}$ and thus also the vanishing of $u^{in}$, which is impossible. This proves the second assertion of Theorem \ref{TH} by applying Lemma \ref{D.1}. The local uniqueness result of Theorem \ref{TH2} follows directly from the second assertion of Theorem \ref{TH}.

  \begin{remark}
   Lemma \ref{D.1} does not hold true if  the curve $\Gamma$ is analytic at $O$. Counterexamples can
be constructed when $\Gamma$ is a line segment or a circle (see  \cite[Remark 3.3]{EH2017}, \cite[Section 4]{LHY} and \cite{CPJ}).
\tcr{ We conjecture that Lemma \ref{D.1} remains valid even under the weaker assumptions: \begin{description}
 \item[] {\rm Conjecture: Let $\Gamma:=\{x_2=f^\pm(x_1),\; x_1\lessgtr 0\}$ where $f^\pm$ are real-analytic functions for
 $x_1\lessgtr 0$. Suppose that $\Gamma$ is $C^\infty$-smooth but not real-analytic at $O$.
  Suppose further that $q_1\neq q_2$ are two (complex) constants and that $w_j$ are two solutions to \eqref{transmission}. Then it holds that $w_1=w_2\equiv 0$ in $B_1$. }
 \end{description}}
In the present paper we only consider a piecewise analytic interface which is not $C^\infty$-smooth at $O$. The proof for $C^\infty$-smooth boundaries with a  non-analytical point requires novel mathematical arguments.
  \end{remark}

\subsection{Important Lemmata}
The proof of Lemma \ref{D.1} will be given in Subsections \ref{sec:2.3}  and \ref{sec:2.4}. In this subsection we prepare several import Lemmata to be used in the proof of Lemma \ref{D.1}.

Setting $w:= w_1-w_2$, it is easy to obtain
\begin{equation}
\label{eq:t1}
\Delta w + q_1 w= (q_1-q_2)w_2 \quad \text{in} \ \   B_1
\end{equation}
subjects to the vanishing Cauchy data
\be\label{Cauchy-data}
 w=\partial_\nu w = 0  \ \ \text{on}\ \Gamma.
\en
It follows from (\ref{Cauchy-data})
that
\begin{equation}
\label{eq:4}
w(x_1,f(x_1))=0,
\end{equation}
\begin{equation}
\label{eq:i2}
\partial_2w(x_1,f(x_1))=0,
\end{equation} for all $x_1\in(-1/2,1/2)$.

Since the potentials $q_j$ are real-analytic, the solutions $w_j$ and $w$ are all analytic functions in $B_1$. Hence,  $w$ and $w_2$ can be expanded into the Taylor series
\be
\label{eq:3}
w(x)=\sum_{i,j\ge 0}\frac{a_{i,j}}{i!j!}x_1^ix_2^j,
\qquad 
w_2(x)=\sum_{i,j\ge 0}\frac{b_{i,j}}{i!j!}x_1^ix_2^j, \quad x=(x_1,x_2) \in B_1.
\en
The Taylor expansion of $\partial_2w$ can be written as
\be
 \label{eq:g3}
\partial_2w(x_1,x_2)={\color {hw}{\sum_{i,j\ge 0}\frac{a_{i,j+1}}{i!j!}x_1^ix_2^j,\qquad x=(x_1,x_2) \in B_1}.}
\en
\tcb{The above Taylor expansions in the Cartesian coordinate system turns out to be convenient in dealing with weakly singular points (see also \cite{LHY}). The corresponding expansions in polar coordinate were used in \cite{ElHu2015} for treating planar corners.} {\color{hw2}Inserting (\ref{eq:6}) into (\ref{eq:3}) and (\ref{eq:g3}), we may rewrite $h(x_1):=w(x_1, f(x_1))$ and $g(x_1):=\partial_2w(x_1, f(x_1))$ as
\be\label{h}
h(x_1)= \sum_{i,j\in \N_0} \frac{a_{i,j}}{i!\, j!} x_1^i [f(x_1)]^j =\left\{\begin{array}{lll}
\sum_{l
\in \N_0} \frac{h^{+}_{l}}{l!}x_1^l, &&\quad -1/2<x_1\leq0, \\
\sum_{l
\in \N_0} \frac{h^{-}_{l}}{l!}x_1^l, &&\quad 0\leq x_1<1/2, \end{array}\right.
\en
and
\be\label{g}
g(x_1)= \sum_{i,j\in \N_0} \frac{a_{i,j+1}}{i!\, j!} x_1^i [f(x_1)]^j =\left\{\begin{array}{lll}
\sum_{l
\in \N_0} \frac{g^{+}_{l}}{l!}x_1^l, &&\quad -1/2<x_1\leq0, \\
\sum_{l
\in \N_0} \frac{g^{-}_{l}}{l!}x_1^l, &&\quad 0\leq x_1<1/2, \end{array}\right.
\en
respectively.}
It then follows from (\ref{eq:4}) and (\ref{eq:i2}) that
\be\label{E1}
h_l^\pm=g_l^\pm=0\quad\mbox{for all}\quad l\in \N_0.
\en
Lemma \ref{D.1} will be proved with the help of (\ref{eq:t1}), (\ref{eq:3}),(\ref{E1}) and the following identity
\be\nonumber
&&\Delta \Delta w +q_1\Delta w + 2\nabla q_1\cdot\nabla w + \Delta q_1 w \\
\label{eq:r9}&=&-(q_1-q_2)q_2w_2+ 2\nabla(q_1-q_2)\cdot\nabla w_2+ \Delta(q_1-q_2)w_2
\en
in $B_1$, which was obtained by taking $\Delta$ on both sides of (\ref{eq:t1}).
In fact, from these relations we shall deduce through the induction argument and the weekly singularity at $O$ (more precisely, $f_m^+\neq f_m^-$) that $a_{i,j}=b_{i,j}=0$ for all $i,j\in\N_0$, which imply the vanishing of $w_1$ and $w_2$ by analyticity.

Let $m\in\N$ ($m\geq 2$) be the order of the singular point $O$ specified in Definition \ref{d0}.  We can always find a number $n\in \N, 2\leq n\leq m$ such that
\be\label{n}
f_{l}=0\quad&\mbox{for all}\quad 0\leq l<n;\qquad
f_n\neq 0,\qquad\mbox{if}\quad n<m.
\en
To prove Lemma \ref{D.1}, we first consider the case of $n<m$. If $n=m$, the proof can be proceeded analogously (see Remark \ref{n=m}).
Since $2\leq n<m$ and $m\geq 2$, we have
\be\label{eq:60}
x_2=f(x_1)=\sum_{j=n}^{m-1} \frac{f_j}{j !}x_1^j+\left\{\begin{array}{lll}
\sum_{j\geq m} \frac{f_j^+}{j !}x_1^j,&&\mbox{if}\quad x_1>0,\\
\sum_{j\geq m} \frac{f_j^-}{j !}x_1^j,&&\mbox{if}\quad x_1<0,
\end{array}\qquad f_{m}^+\neq f_{m}^-.\right.
\en
Below we state two important lemmata , whose proof will be postponed to the appendix. The first one describes the relation between the coefficients $h_l^\pm$, $g_l^\pm$ and $a_{i,j}$ for $0\leq l<3n-1$. Recall from and \eqref{h} and \eqref{eq:60} that
\be\label{h2}
h(x_1)=\sum_{i,j\in \N_0} \frac{a_{i,j}}{i !\,j !} x_1^i\left( \sum_{l=n}^{m-1} \frac{f_l}{j!} x_1^l  +   \sum_{l=m}^{\infty} \frac{f_l^\pm}{j!} x_1^l  \right)^j=\sum_{l\in \N_0} \frac{h_l^\pm}{l\, ! } x_1^l\quad \mbox{for}\quad x_1\lessgtr 0.
\en
\begin{lemma}\begin{itemize}
\item[(i)] For $0\leq l\leq n-1$, we have
\be\label{E2}
h_l^\pm=\frac{a_{l,0}}{l\,!},\quad g_l^\pm=\frac{a_{l,1}}{l\,!}.
\en
\item[(ii)] For $n\leq l\leq 2n-1$, we have
\be\label{E3}
h_l^\pm=\frac{a_{l,0}}{l\,!}+\sum_{j=0}^{l-n} \frac{a_{l-n-j,1}}{(l-n-j)\,!}\frac{f^\pm_{j+n}}{(j+n)\,!},\\ \label{E4}
g_l^\pm=\frac{a_{l,1}}{l\,!}+\sum_{j=0}^{l-n} \frac{a_{l-n-j,2}}{(l-n-j)\,!}\frac{f^\pm_{j+n}}{(j+n)\,!}.
\en
\item[(iii)] For $2n\leq l\leq 3n-1$, we have
\be\label{E5}
h_l^\pm=\frac{a_{l,0}}{l\,!}+\sum_{j=0}^{l-n} \frac{a_{l-n-j,1}}{(l-n-j)\,!}\frac{f^\pm_{j+n}}{(j+n)\,!}+\sum_{j=0}^{l-2n} \frac{a_{l-2n-j,2}}{(l-2n-j)\,!}
\left(\sum_{i=0}^j \frac{f^\pm_{i+n}\; f^\pm_{j-i+n}}{(i+n)!\;(j-i+n)\,!} \right)
\en
and $g_l^\pm$ takes the same form as $h_l^\pm$ with $a_{\cdot, j}$ replaced by $a_{\cdot, j+1}$.
\end{itemize}
\end{lemma}
The above lemma follows directly by equating the coefficients for $x_1^l$ on both sides of \eqref{h2} and using the fact that $[f(x_1)]^j=O(x_1^{nj})$.  Writhe the index $l=i+nj$ for$ i, j, n \in \N_0$. Then, the relation $0\leq l\leq n-1$ implies $j=0, i=l$ in the first assertion; the relation $n\leq l\leq 2n-1$ implies $j=0, i=l$ or $j=1, i=l-n$ in the second assertion; the relation
$2n\leq l\leq 3n-1$ implies $j=0, i=l$, or  $j=1, i=l-n$, or $j=2, i=l-2n$ in the third assertion.
These results will be used in justifying the initial steps of the induction hypothesis.
It is not necessary to calculate $h_l^\pm$ for all $l\in \N_0$. In our induction arguments, we need the following definition and lemma.

\begin{definition}\label{d}
Let $n$ be given by (\ref{n}). For $a\in \N_0$ and $\alpha \in \{0,1,2\}$, it is said that the  pair  $(i,j)\in  \N_0 \times  \N_0$ with $i\le a$ and $j\ge \alpha$ belongs to the index set $\Upsilon_{\alpha}^a$ if either $(i,j)=(a,\alpha)$ or there exit some $d \in \N$ and two sequences of positive integers ${\{i_k\}}^d_{k=1}, {\{j_k\}}^d_{k=1}$ with $i_k\ge n$ such that
\begin{equation}\label{e1}
i+\sum^{d}_{k=1} i_kj_k=a,\qquad \sum^{d}_{k=1}j_k=j-\alpha.
\end{equation}
\end{definition}

\begin{lemma}\label{Le:1}
 Suppose that the coefficients in the Taylor expansion of $w$ (see (\ref{eq:3})) fulfill the relations
 \begin{align}\label{er}
 a_{i,j}=0\qquad\mbox{if}\quad i+jn\le k-1+3n\quad\mbox{or} \quad i+j \le k+2,\; j>3,
 \end{align}
 for some $k\in \N$. Then the relation $g(x_1)\equiv 0$ for all $x_1\in(-1/2,1/2)$ (see (\ref{g}) for the definition of $g$) implies that 
\be
\label{eq:16}
\frac{{(f^{\pm}_m)}^2}{{(m!)}^2}\frac{3a_{k,3}}{k!3!}+\frac{{f^{\pm}_m}}{{m!}} 2D_{k,1}+D_{k,2}=0.
\en


The relation $h(x_1)=0$ for all $x_1\in(-1/2,1/2)$ implies that
\be
\label{eq:15}
\frac{{(f^{\pm}_m)}^3}{{(m!)}^3}\frac{a_{k,3}}{k!3!}+\frac{{(f^{\pm}_m)}^2}{{(m!)}^2} D_{k,1}
+\frac{f_{\pm}^{m}}{{m!}}\, D_{k,2}
+ D_{k, 3}=0.
\en
Here $D_{k, l}\in \C$ \tcr{($l=1,2,3$)} depends on $a_{i,j}$ with $(i,j) \in \Upsilon^{lm+k}_{3-l}$ and only on the coefficients $f_j$ of $f(x_1)$ with $j<m$.
\end{lemma}

%
%





To prove Lemma \ref{D.1} when $n<m$, we need to consider two cases:

Case (i): $q_1(O) \ne q_2(O)$;

Case (ii): $q_1(O)=q_2(O), \partial_{1}q_1(O) \ne \partial_{1}q_2(O)$.

\subsection{Proof of Lemma \ref{D.1} in Case (i): $q_1(O) \ne q_2(O)$.}\label{sec:2.3}
 For simplicity, we shall divide our proof into four steps. Recall again from \eqref{E1} that $h_l^\pm$ and $g_l^\pm$ vanish for all $l\in \N_0$.

\textbf{Step 1}: First, it follows from the expressions of $h_l^\pm$ and $g_l^\pm$ for $0\leq l\leq n-1$ (see \eqref{E2}) that
\[ a_{i,0}=0,\qquad0\leq i\leq n-1.
 \] and
\begin{equation}\label{t1}
a_{i,1}=0, \qquad 1\le i \le n-1.
\end{equation}
{\color{hw} Inserting \eqref{t1} into \eqref{E3}
yields
\[
a_{i+n,0}=0, \qquad 0\le i\le n-1.
\]}
To summarize above, we get
\ben
a_{i,j}=0, \qquad  i+jn\leq 2n-1.
\enn

 \textbf{Step 2}: We shall prove $a_{i,j}=0, 2n\leq i+jn \le 3n-1$ and $b_{0,0}=0$.

Using the results of Step 1 and the expression of $g(x_1)$, we get
\be\label{E6}
0=\sum_{i\geq n} \frac{a_{i, 1}}{ i\, !} x_1^i + \left(\sum_{i\geq 0} \frac{a_{i, 2}}{ i\, !\; } x_1^i\right)\; \left( \sum_{l=n}^{m-1} \frac{f_l}{l!} x_1^l  +   \sum_{l=m}^{\infty} \frac{f_l^\pm}{l!} x_1^l   \right)+\cdots.
\en
Equating the coefficients of $x_1^{\gamma}$ with $m\leq \gamma\leq n+m-1$, we have
\begin{align*}
\eta_{\gamma}+\frac{1}{({\gamma}-m)!}\frac{f_{m}^{+}}{m!}a_{{\gamma}-m,2}=\eta_{\gamma} +\frac{1}{({\gamma}-m)!}\frac{f_{m}^{-}}{m!}a_{{\gamma}-m,2},
\end{align*}
where $\eta_{\gamma}$ depends on $a_{i, j}$ and $f_l$ with $l<m$.
Utilizing the condition that $f_{m}^{+}\ne f_{m}^{-}$, we have
\begin{align*}
a_{{\gamma}-m,2}=0\quad \mbox{for any}\quad m\leq \gamma\leq n+m-1,
\end{align*}
implying that $a_{l, 2}=0$ for all $0\leq l\leq n-1$.
Now, using \eqref{E4} we get
$a_{l, 1}=0$ for all $n\leq l\leq 2n-1$. Together with \eqref{E5}, this gives
$a_{l, 0}=0$ for all $2n\leq l\leq 3n-1$.
Now we conclude that $a_{i,j}=0$ for all $i,j$ such that $2n\leq i+jn \le 3n-1.$

The results in the first two steps give rise to $a_{0,0}=a_{2,0}=a_{0,2}=0$, implying that $w(O)=\Delta w(O)=0$.
Since $q_1(O) \ne q_2(O)$, it is deduced from (\ref{eq:t1})  that $w_2(O)=0$, that is,
\begin{equation}
\label{T:5}
b_{0,0}=0.
\end{equation}
%

\textbf{Step 3}: In this step, we will prove 
\be\label{N:10}\begin{split}
a_{i,j}=0\quad &\mbox{if}\quad 3n \le i+jn\le 3n+1 \quad\mbox{or}\quad  i+j = 4,\\
b_{i,j}=0\quad &\mbox{if}\quad  1 \le i+j \le 2.
\end{split}
\en
We first  consider the case of $i+nj=3n$.  
It is deduced from Steps 1-2 that that coefficients $a_{i,j}$ satisfy the assumption of Lemma \ref{Le:1} with $k = 0$. Taking $k=0$ in  (\ref{eq:16})  and (\ref{eq:15}), we obtain
\begin{align}
\label{eq:45}
&{(f^{\pm}_m)}^3\frac{a_{0,3}}{3!{(m!)}^3}+\frac{{(f^{\pm}_m)}^2}{{(m!)}^2} D_{0,1}+ \frac{f_{\pm}^{m}}{{m!}}D_{0,2} +D_{0,3}=0, \\
\label{eq:46}
&{(f^{\pm}_m)}^2\frac{3a_{0,3}}{3!{(m!)}^2}+2\frac{{f^{\pm}_m}}{{m!}}D_{0,1}+D_{0,2}=0,
\end{align}
Since $f_m^+\neq f_m^-$, we deduce from
(\ref{eq:46}) and (\ref{eq:45}) that
\be\label{T:6}
{(f^{+}_m)}^2\frac{3a_{0,3}}{3!{(m!)}^2}
+2\frac{{f^{+}_m}}{{m!}}D_{0,1}&=& {(f_m^{-})}^2\frac{3a_{0,3}}{3!{(m!)}^2}
+2\frac{{f_m^{-}}}{{m!}}D_{0,1},
\\ \label{T:7}
{2(f^{+}_m)}^3\frac{a_{0,3}}{3!{(m!)}^3}+\frac{{(f^{+}_m)}^2}{{(m!)}^2} D_{0,1}&=& {2(f_m^{-})}^3\frac{a_{0,3}}{3!{(m!)}^3}+\frac{{(f_m^{-})}^2}{{(m!)}^2} D_{0,1}.
\en
Combing (\ref{T:6}) and (\ref{T:7}) gives the algebraic equations for $a_{0,3}$ and $D_{0,1}$,
\ben 
&&{(f^{+}_m+f_{m}^{-}})\frac{3a_{0,3}}{3!{m!}}+2{D_{0,1}}=0,
\\ \label{T:9}
&&\frac{2a_{0,3}}{3!{m!}}({(f_{m}^{+})}^2
+{(f_{m}^{-})}^2+{f_{m}^{+}}{f_{m}^{-}})
+D_{0,1}({f_{m}^{+}}+{f_{m}^{-}})=0,
\enn
which can be written in the matrix form
\be
M\begin{bmatrix}\frac{a_{0,3}}{3!{m!}} \\ D_{0,1}\end{bmatrix}=0,\quad
M=:
\begin{bmatrix}
{3(f_{m}^{+}+f_{m}^{-})}, & 2 \\
2({(f_{m}^{+})}^2+{(f_{m}^{-})}^2+{f_{m}^{+}}{f_{m}^{-}}), &({f_{m}^{+}}+{f_{m}^{-}})
\end{bmatrix}.
\en
Since $|M| =-(f^{+}_m-f_m^{-})^2\ne 0$, we obtain
\be\label{T:10}
a_{0,3}=0.
\en
Now, the expression of $g$ can be rephrased in $x_1\lessgtr 0$  as (cf. \eqref{E6})
\begin{align*}
0=&\sum_{i\geq 2n} \frac{a_{i, 1}}{ i\, !} x_1^i + \left(\sum_{i\geq n} \frac{a_{i, 2}}{ i\, !} x_1^i\right)\; \left( \sum_{l=n}^{m-1} \frac{f_l}{l!} x_1^l  +   \sum_{l=m}^{\infty} \frac{f_l^\pm}{l!} x_1^l   \right)\\
&+\left(\sum_{i\geq 0} \frac{a_{i, 3}}{ i\, !\; 2 !} x_1^i\right)\; \left( \sum_{l=n}^{m-1} \frac{f_l}{l!} x_1^l  +   \sum_{l=m}^{\infty} \frac{f_l^\pm}{l!} x_1^l   \right)^2+\cdots.
\end{align*}
where we have used the results of Step 2.
Making the difference for the coefficients of $x_1^{n+m}$,  we get
\begin{align*} 
\frac{a_{0,3}}{2!}\frac{f_{n}}{n!}\frac{f_{m}^{+}}{m!}+\frac{a_{n,2}}{n! 2!}\frac{f_{m}^{+}}{m!}=\frac{a_{0,3}}{2!}\frac{f_{n}}{n!}\frac{f_{m}^{-}}{m!}+\frac{a_{n,2}}{n! 2!}\frac{f_{m}^{-}}{m!}.
\end{align*}
Utilizing the fact $f_{m}^{+}\ne f_{m}^{-}$, we get
\begin{align*}
\frac{a_{0,3}}{2!}\frac{f_{n}}{n! 2!}+\frac{a_{n,2}}{n! 2!}=0.
\end{align*}
This together with (\ref{T:10}) and $f_n\neq 0$ yields that
\be\label{T:11}
a_{n,2}=0.
\en
Furthermore,  by  equating the coefficient of $x_1^{2n}$ we get
\begin{equation*}
\frac{a_{0,3}}{2!(n!)^2}{(f_{n})}^2
+\frac{a_{n,2}}{(n!)^2}f_{n}
+\frac{a_{2n,1}}{(2n)!}=0.
\end{equation*}
On the other hand,  equating the coefficients of $x_1^{3n}$ in the expression of $h(x_1)$ leads to
\begin{equation*}
\frac{a_{0,3}}{3!{n!}^3}{(f_{n})}^3+\frac{a_{n,2}}{2{(n!)}^3}{(f_{n})}^2
+\frac{a_{2n,1}}{(2n)!n!}f_n+\frac{a_{3n,0}}{(3n)!}=0.
\end{equation*}
Combining (\ref{T:10}),(\ref{T:11}) and the previous two identities leads to
$a_{n,2}=a_{3n,0}=0$. This proves the first relations for $a_{i,j}$ with $i+jn=3n$ appearing in (\ref{N:10}), i.e.,
\be\label{T:15}
a_{0,3}=a_{n,2}=a_{2n,1}=a_{3n,0}=0.
\en
Recall from the first two steps that $a_{0,3}=a_{2,1}=a_{1,2}=0$. Hence, $\partial_j\Delta w(O)=0$ for $j=1,2$. Taking $\partial_j$ to both sides of (\ref{eq:t1}) and using the fact that
$q_1(O)\ne q_2(O)$ and $w_2(O)=0$, we get $\partial_j w_2(O)=0$, or equivalently,
\begin{equation}
\label{T:12}
b_{1,0}=b_{0,1}=0.
\end{equation}
Repeating the same arguments, one can prove for $i+jn=3n+1$ that
\begin{equation}
\label{T:13}
a_{1,3}=a_{n+1,2}=a_{2n+1,1}=a_{3n+1,0}=0.
\end{equation}
This together $a_{l,0}=0$ for $2n\leq l\leq 3n-1$ and (\ref{T:15}) yields that
\be\label{N:11}
a_{4,0}=a_{3,1}=a_{2,2}= a_{1,3}=0.
\en
Equating the coefficients of the lowest order in (\ref{eq:r9}) and using (\ref{T:5}), (\ref{T:12}), we readily obtain $a_{0,4}=0$. Furthermore, we get $b_{0,2}=b_{1,1}=b_{2,0}=0$ with the help of (\ref{eq:t1}). This together with (\ref{T:15}),(\ref{T:12}),(\ref{T:13}) and (\ref{N:11}) gives (\ref{N:10}).


%


\textbf{Step 4}: Induction arguments. We make the hypothesis that
\be\label{E8}
 a_{i,j}&=&0\quad\mbox{for all}\;(i,j)\in \N_0\times\N_0\;\mbox{such that} \;\left\{\begin{array}{lll} i+jn\le p-1 &&\mbox{if}\; j\le 3, \\
  i+j \le p-3n+2,\quad&&\mbox{if}\; j\ge 4,
 \end{array}\right.
 \\ \nonumber
 b_{i,j}&=&0\quad\mbox{for all}\;(i,j)\in \N_0\times\N_0\;\mbox{such that}\; i+j\le p-3n
\en
for some $p\ge 3n+2$, $p\in\mathds{N}$. Note that for $p=3n+2$, this hypothesis has been proved in steps 1-3.
 Now we need to prove the hypothesis with the index $p$ replaced by $p+1$. For this purpose, it suffices to check that
 \be\label{eq:re3}
 a_{i,j}&=&0\quad\mbox{for all}\;(i,j)\in \N_0\times\N_0\;\mbox{such that} \;\left\{\begin{array}{lll} i+jn= p &&\mbox{if}\; j\le 3, \\
  i+j = p-3n+3,\quad&&\mbox{if}\; j\ge 4,
 \end{array}\right.
 \\ \label{eq:t17}
 b_{i,j}&=&0\quad\mbox{for all}\;(i,j)\in \N_0\times\N_0\;\mbox{such that}\; i+j= p-3n+1.
\en

For notational convenience, we introduce the set
\[
I_p:=\{(i,j)\in \mathbb{N}_0\times \mathbb{N}_0:\; i+jn=p\}.
\]
By our assumption that $n>2$, it holds for $(i_1,j_1), (i_2,j_2) \in I_p$ that
\[
i_1+j_1>i_2+j_2, \quad \mbox{if} \quad j_1<j_2.
\]
Therefore, for $(i,j) \in I_p$, we have
\begin{equation*}
i+j<p-3n+3,\quad\mbox{if}\quad j>3.
\end{equation*}
By the induction hypothesis (\ref{E8}), we get
\begin{equation*}
a_{i,j}=0,\quad \mbox{for all}\; (i,j) \in I_p,\; j>3.
\end{equation*}
Furthermore, it follows from the induction hypothesis that the coefficients $a_{i,j}$ fulfill the assumption of Lemma \ref{Le:1} with $k = p-3n$. Hence, setting $k=p-3n$ in  (\ref{eq:16}) and (\ref{eq:15}) yields
\be
\label{eq:69}
0&=&{(f^{\pm}_m)}^3\frac{a_{p-3n,3}}{(p-3n)!3!{(m!)}^3}+\frac{{(f^{\pm}_m)}^2}{{(m!)}^2} D_{p-3n,1}+ \frac{f_{\pm}^{m}}{{m!}}D_{p-3n, 2}+D_{p-3n, 3}, 
\\
\label{eq:70}
0&=&{(f^{\pm}_m)}^2\frac{3a_{p-3n,3}}{(p-3n)!3!{(m!)}^2}
+2\frac{{f^{\pm}_m}}{{m!}}D_{p-3n, 1}+D_{p-3n, 2}.
\en
Similarly to the derivation of (\ref{T:9}) and (\ref{T:10}) in Step 3, using $f_m^+\neq f_m^-$ we can get  a linear algebraic system for $a_{p-3n,3}$ and $D_{p-3n,1}$ as follows:
\be{\label{T:18}}
M\,\begin{pmatrix}
\frac{a_{p-3n,3}}{(p-3n)!3!m!}\\ D_{p-3n, 1}\end{pmatrix}=0,
\en
where $M\in\R^{2\times 2}$ is defined again by (\ref{T:9}).
The fact that $|M|\ne 0$ gives $a_{p-3n,3}=0$.
Inserting (\ref{eq:6}) into (\ref{eq:i2}) and equating the coefficients of $x_1^{p-2n+m}$, we readily get
\begin{align*}
2\frac{3a_{p-3n,3}}{(p-3n)!3!}\frac{f_{n}}{n!}\frac{f_{m}^{+}}{m!}+\frac{2a_{p-2n,2}}{2(p-2n)!}\frac{f_{m}^{+}}{m!}=2\frac{3a_{p-3n,3}}{(p-3n)!3!}\frac{f_{n}}{n!}\frac{f_{m}^{-}}{m!}+\frac{2a_{p-2n,2}}{2(p-2n)!}\frac{f_{m}^{-}}{m!}
\end{align*}
This combined with $a_{p-3n,3}=0$ and $f_m^+\neq f_m^-$ yields \be\label{T:17}
a_{p-3n,3}= a_{p-2n,2}=0.
\en
Using the induction hypothesis and comparing the coefficients of $x_1^{p}$ and $x_1^{p-n}$, respectively, it follows that
\be
\label{eq:66}
\frac{a_{p-3n,3}}{3!(p-3n)!}\frac{{(f_n)}^3}{{(n!)}^3}
+\frac{a_{p-2n,2}}{2(p-2n)!}\frac{{(f_n)}^2}{{(n!)}^2}
+\frac{a_{p-n,1}}{(p-n)!}\frac{{f_n}}{{n!}}
+\frac{a_{p,0}}{(p)!}=0,\\
\label{eq:67}
\frac{3a_{p-3n,3}}{3!(p-3n)!}\frac{{(f_n)}^2}{{(n!)}^2}+\frac{2a_{p-2n,2}}{2(p-2n)!{(n!)}}\frac{{f_n}}{{n!}}+\frac{a_{p-n,1}}{(p-n)!}=0.
\en
Combing (\ref{T:17}), (\ref{eq:66}) and (\ref{eq:67}), we have
\begin{equation}
\label{eq:g12}
a_{p-3n,3}=a_{p-2n,2}=a_{p-n,1}=a_{p,0}=0.
\end{equation}
On the other hand, utilizing the fact that $n\geq 2$ and
the induction hypothesis for $a_{i,j}$ with $j\leq 3$, we get
\begin{equation}
\label{eq:dd}
a_{p-3n+1,2}=a_{p-3n+2,1}=a_{p-3n+3,0}=0.
\end{equation}
With the aid of (\ref{eq:r9}),(\ref{eq:dd}) and the induction hypothesis $b_{i,j}=0, i+j\le p-3n$, we readily get by equating the coefficients of $x_1^{i_1}x_2^{j_1}$,
with $i_1+j_1 =p-3n-1$ that
\begin{equation} \label{eq:202}
a_{i,j}=0,\quad\mbox{if}\quad i+j =p-3n+3,\quad j\ge 4.
\end{equation}
This together with (\ref{eq:t1}) gives  $b_{i,j}=0$ if $i+j= p-3n+1$.

By far we have proved all relations in (\ref{eq:re3}) and (\ref{eq:t17}).
By induction, it holds that $a_{i,j}=b_{i,j}=0$ for all $i,j \in \mathbb{N}_0$. The proof of case (i) is thus complete. \hfill $\Box$

\subsection{Proof of Lemma \ref{D.1} in Case (ii): $q_1(O)= q_2(O), \partial_{1} q_1(O) \ne \partial_{1} q_2(O)$.}\label{sec:2.4}
The proof in the second case can be carried out analogously to case (i). Below we sketch the proof by indicating the differences to case (i).

\textbf{Step 1}: Using the same arguments in the proof of case (i), we have $a_{i,j}=0$, $i+jn \le 3n-1$.

\textbf{Step 2} : Similar to the derivation of (\ref{T:15}) in case (i), we can obtain  $a_{0,3}=a_{n,2}=a_{2n,1}=a_{3n,0}=0. $ From $(\ref{eq:t1})$, we get
\ben
\partial_{1} (q_1(O)-q_2(O))b_{0,0} = a_{3,0}=0.
\enn
This together with the condition $\partial_{1} q_1(O) \ne \partial_{1} q_2(O)$ gives $b_{0,0}= 0$. Repeating this procedure, we could prove $a_{1,3}=a_{n+1,2}=a_{2n+1,1}=a_{3n+1,0}=0$. Combing this with (\ref{eq:t1}) yields that
\ben
&&\big(\partial_{1} (q_1(O)-q_2(O))\big)b_{1,0} = \frac{1}{2}a_{4,0}+ \frac{1}{2}a_{2,2}=0,\notag\\
&& \big(\partial_{2} (q_1(O)-q_2(O))\big)b_{1,0} + \big(\partial_{1} (q_1(O)-q_2(O))\big)b_{0,1} =  a_{3,1}+a_{1,3} =0,\\
&& \big(\partial_{2} (q_1(O)-q_2(O))\big)b_{0,1}  = \frac{1}{2}a_{0,4},\notag
\enn
which imply $b_{1,0} = b_{0,1} = 0$. This together with (\ref{eq:r9}) leads to $a_{0,4} = 0$. To summarize this step we obtain
\begin{align*}
&a_{i,j}=0, \quad 3n \le i+jn\le 3n+1 \quad\mbox{or}\quad i+j = 4; \\
&b_{i,j}=0, \quad i+j\le 1.
\end{align*}

\textbf{Step 3}. In this step, we will adopt an
induction argument
 similar to Step 4 of case (i). We make the same hypothesis for $a_{i,j}$ as before:  $a_{i,j}=0$ for all $i+jn\le p-1, j\le 3$ and $i+j \le p-3n+2, j\ge 4$, where $p\ge 3n+2$, $ p\in\mathbb{N}_0$. However, we assume that $b_{i,j}=0$, $i+j\le p-3n-1$, with the bound of $i+j$ different from case (i).
 Our aim is to prove (\ref{eq:re3}) and
 \begin{align} \label{N:13}
 &b_{i,j}= 0, i+j= p-3n.
 \end{align}
 We remark that the relations in (\ref{eq:re3}) can be proved in the same way as Step 3 of case (i). To prove (\ref{N:13}), with the help of induction hypothesis, we conclude from (\ref{eq:t1}) that
\begin{align}
&\frac{\big(\partial_{1} (q_1-q_2)(O)\big)}{(p_1-3)!}b_{p_1-3,0} \notag
=\frac{1}{(p_1-2)!}(a_{p_1,0}+ a_{p_1-2,2}),\\
&\frac{\big(\partial_{2} (q_1-q_2)(O)\big)}{(p_1-k_1-3)!k_1!}b_{p_1-k_1-3,k_1} + \frac{\big(\partial_{1} (q_1-q_2)(O)\big)}{(p_1-k_1-4)!(k_1+1)!}b_{p_1-k_1-4,k_1+1} \label{e6}\\
&=\frac{1}{(p_1-k_1-3)!(k_1+1)!}(a_{p_1-k_1-1,k_1+1}+a_{p_1-k_1-3,k_1+3}),\label{N:3} \\
&\frac{\big(\partial_{2} (q_1-q_2)(O)\big)}{(p_1-3)!}b_{0,p_1-3}
 =\frac{1}{(p_1-2)!}a_{2,p_1-2}, \notag
\end{align}
where $p_1= p-3n+3$ and $k_1,k_2$ are two integers satisfying $ 0\le k_1 \le(p_1-4), 0\le k_2 \le (p_1-4)$. Analogously, combing with induction hypothesis and (\ref{eq:r9}) gives that
\begin{align}
&\frac{2}{(p_1-4)!}\bigg(\big(\partial_{1} (q_1-q_2)(O)\big)b_{p_1-3,0} + \big(\partial_{2} (q_1-q_2)(O)\big)b_{p_1-4,1}\bigg)
=\frac{(a_{p_1,0}+ 2a_{p_1-2,2} + a_{p_1-4,4})}{(p_1-4)!},\notag\\
&\frac{2}{(p_1-k_2-4)!k_2!}\bigg(\big(\partial_{1} (q_1(O)-q_2(O))\big)b_{p_1-k_2-3,k_2},
+ \big(\partial_{2} (q_1(O)-q_2(O))\big)b_{p_1-k_2-4,k_2+1}\bigg) \notag\\
&= \frac{1}{(p_1-k_2-4)!k_2!}(a_{p_1-k_2,k_2} + 2a_{p_1-k_2-2,k_2+2} + a_{p_1-k_2-4,k_2+4}).\label{N:5}
\end{align}
 Further, using the similar arguments in deriving (\ref{eq:g12}), we readily obtain that
\[
a_{p_1,0}=a_{p_1-1,1}= a_{p_1-2,2}= a_{p_1-3,3}.
\]
This together with (\ref{e6}) and (\ref{N:3}) with $k_1 = 0$ implies
\be\label{N:12}
b_{p_1-3,0}=b_{p_1-4,1}=0.
\en
Hence, it is deduced from  (\ref{N:5}) and (\ref{N:12}) that $a_{p_1-4,4} = 0$, and combining (\ref{N:3}) with $k_1 = 1$ yields that $b_{p_1-5,2}=0$. Setting $k_2=1$ in (\ref{N:5}), it is esay to verify that $a_{p_1-5,5}= 0$ due to $b_{p_1-4,1}=b_{p_1-5,2}=0$. Repeating this procedure successively, we will get (\ref{eq:re3}) and (\ref{N:13}).   

 By induction, it holds that $a_{i,j}= b_{i,j}=0$ for all $ i,j \in \mathbb{N}_0$. The proof of Lemma \ref{D.1} is thus complete in the second case. \hfill $\Box$

\begin{re}\label{n=m}
In the case of $n=m$, the proof of  Lemma \ref{D.1} should be slightly modified.
The only difference is to replace (\ref{eq:69}) and (\ref{eq:70}) by
\be{\label{T:16}}
{(f^{\pm}_m)}^3\frac{a_{p-3m,3}}{(p-3m)!3!{(m!)}^3}+ \frac{a_{p-2m,2}}{(p-2m)!2}\frac{{(f^{\pm}_m)}^2}{{(m!)}^2}+ \frac{a_{p-m,1}}{(p-m)!} \frac{f_{\pm}^{m}}{{m!}} + \frac{a_{p,0}}{p!}=0
\en
and
\be{\label{T:20}}
{(f^{\pm}_m)}^2\frac{3a_{p-3m,3}}{(p-3m)!3!{(m!)}^2}+ \frac{2a_{p-2m,2}}{(p-2m)!2}\frac{{(f^{\pm}_m)}}{{(m!)}}+ \frac{a_{p-m,1}}{(p-m)!}=0,
\en
respectively.
Similar to the derivation of (\ref{T:18}), we can obtain
$a_{p-3m,3}=a_{p-2m,2}=0$. This together with (\ref{T:16}) and (\ref{T:20}) also gives that
\[
a_{p-3m,3}=a_{p-2m,2}=a_{p-m,1}=a_{p,0}=0.
\]
Proceeding with the same lines as for the case $n<m$,
 we can also prove $w_1=w_2\equiv0$ when $n=m$.
\end{re}

\section{Characterization of radiating sources}\label{sec:Source}

 This section is devoted to proving Theorems \ref{th1} and \ref{th2}.
 One should note that the inverse source problem for recovering a source term is linear, whereas the inverse medium problems for shape identification and medium recovery are both nonlinear. Hence, the techniques for extracting source information from measurement data are usually easier than inverse medium scattering problems.
Lemma \ref{Le:4} below can be regarded as the analogue of Lemma \ref{D.1} for inverse source problems.

\begin{lemma} \label{Le:4}
Let $w\in H^1(B_1)$ be a solution to
\be\left\{\begin{array}{lll}
\Delta w +\kappa^2 \mathfrak{n}(x)w=s \quad &&\mbox{in}\quad B_1, \label{N:7}\\
w=0, \frac{\partial w}{\partial \nu}=0 \quad &&\mbox{on} \quad \Gamma\ \cap B_1,
\end{array}\right.
\en
where $\mathfrak{n}(x)$ is analytic in $B_1$ and $O\in \Gamma$ is a weakly singular point. The source term $s(x)$ is supposed to satisfy the elliptic equation
\be\label{equation:s}
\Delta s(x) + A(x)\cdot \nabla s(x) + b(x) s(x)=0 \ \text{in} \ B_1 ,
\en
where $A(x)=(a_1(x), a_2(x))$ and $b(x)$ are both analytic in $B_1$. Then
\be \label{N:9}
w=s\equiv0 \ \text{in} \   B_1.
\en
\end{lemma}
\begin{proof}
Since $w$ and $s$ are real-analytic in $B_1$, they can be expanded into the Taylor series
\be\label{TE}
w(x)=\sum_{i,j\ge 0}\frac{w_{i,j}}{i!j!}x_1^ix_2^j, \quad
s(x)= \sum_{i,j\ge 0}\frac{s_{i,j}}{i!j!}x_1^ix_2^j,\quad x=(x_1,x_2) \in B_1.
\en
Taking $\Delta$ to the equation of $w$ and using the governing equation of $s$, we find
\be \label{N:8}
\Delta\Delta w +\kappa^2\Delta (\mathfrak{n}(x)w) = -A(x)\cdot \nabla s(x) - b(x) s(x).
\en
From (\ref{N:7}), it follows that, for any integer ${\color{red}l}\geq 2$,
 the statement  $w_{i,j}= 0$, $i+j \le l$ leads to $s_{i,j}=0, i+j \le l-2$. Further, similar to the derivation of (\ref{eq:202}), one can deduce from (\ref{N:8}) and the relations  $w_{i,j}=0, i+j=l,j\le 3$; $w_{i,j} = 0 , i+j \le l-1$ and  $s_{i,j}=0, i+j\le l-3$ that $w_{i,j} = 0$ , $i+j = l, j\ge 4$.  Hence, using the same method as employed in the proof of Lemma \ref{D.1}, one can prove (\ref{N:9}).
 We omit the details for brevity.
\end{proof}
\begin{re}
Lemma \ref{Le:4} applies to  analytical source terms $s(x)$ whose lowest order Taylor expansion at $O$ is harmonic. By \cite{HL}, the solutions to  (\ref{equation:s}) process such a property.
\end{re}

\begin{lemma} \label{Le:5}
If the source term $s$ is only required to be analytic in Lemma \ref{Le:4}, then
\ben 
s(O)=|\nabla s (O)|=0.
\enn
\end{lemma}
\begin{proof}
The analyticity of $\mathfrak{n}$ and $s$ guarantees the same Taylor expansions as in (\ref{TE}). Employing the same arguments in steps 1-2 in the proof of Lemma \ref{D.1} yields $s(O)=0$. The method for proving (\ref{T:15}) in the proof of Lemma \ref{D.1}  could directly lead to $|\nabla s(O)|=0$.
\end{proof}
Now we are ready to prove Theorems \ref{th1} and \ref{th2}, by applying Lemmata \ref{Le:5} and \ref{Le:4}, respectively.

{\bf Proof of Theorem \ref{th1}:}
\begin{itemize}
\item[(i)]
Assuming $v^{\infty} = 0$, then we obtain $v\equiv0$ in $|x|>R$ by Rellich's lemma and
 $v\equiv0$ in
 $\R^2\backslash\overline{D}$ by the unique continuation of the Helmholtz equation (see  \cite[Theorem 17.2.6, Chapter XVII]{UCP}). In particular, the Cauchy data $v$,$\partial_\nu v $ vanish on $\Gamma_{\epsilon} = \partial D \cap B_{\epsilon}(O)$ for some $\epsilon > 0$.
 Since $\Gamma_\epsilon$ is piecewise analytic and the Cauchy data are both analytic on $\Gamma_\epsilon$, by the Cauchy-Kovalevski theorem we can extend $v$ from $D\cap B_\epsilon(O)$ to $B_\epsilon$. Hence, the extended function $v$ satisfies
\be \label{N:14}
\begin{cases}
&\Delta v +\kappa^2 \mathfrak{n}(x)v=S \quad \text{in} \ B_{\epsilon},\\
&v=0, \frac{\partial v}{\partial \nu}=0 \qquad \quad\, \text{on} \ \Gamma_{\epsilon},\\
\end{cases}
\en
where $S$ is the analytical extension of $s$ around $O$.
Applying Lemma \ref{Le:5} gives $s(O)=|\nabla s (O)|=0$, which contradicts the assumption that $|s(O)|+|\nabla s(O)| >0$.
\item[(ii)] Suppose that $v$ can be analytically continued from $B_R\backslash\overline{D}$ to $B_\epsilon(O)$ for some $\epsilon>0$. The extended solution obviously satisfies $\Delta v +\kappa^2 \mathfrak{n}(x)v=0$ in $B_\epsilon(O)$. On the other hand, by  Lemma \ref{nle} 
    we can extend $v$ from $D\cap B_\epsilon(O)$ to $B_\epsilon$ as the solution of $\Delta w +\kappa^2 \mathfrak{n}(x)w=S$ in $B_\epsilon(O)$. Now, we can observe that the difference $w-v$ satisfies the same Cauchy problem as in (\ref{N:14}). Applying Lemma \ref{Le:5} yields the same contradiction to $|s(O)|+|\nabla s(O)| >0$.

\end{itemize}
\hfill$\Box$


{\bf Proof of Theorem \ref{th2}:}
\begin{itemize}
\item[(i)] Suppose that there are two sources $s$ and $\tilde s$ which generate identical far-field patterns and have the same support $D$. Denote by $v$ and $\tilde v$ the wave fields radiated by $s$ and $\tilde s$, respectively. By Rellich's lemma and the unique continuation, we know  $v=\tilde v$ in $\R^2 \backslash\overline{D}$. Setting $u:= v-\tilde v$, it follows that
\be \label{N:16}
\begin{cases}
&\Delta u +\kappa^2 \mathfrak{n}(x)u= s-\tilde{s} \quad \text{in} \ B_{\epsilon},\\
&u=0, \frac{\partial u}{\partial \nu}=0 \qquad\qquad\quad \text{on} \ \partial D \cap B_{\epsilon},\\
\end{cases}
\en
for some $\epsilon >0$. Applying Lemma  $\ref{Le:5}$ gives $s(O)=\tilde{s}(O)$ and $\nabla s(O)=\nabla\tilde{s}(O)$.
\item[(ii)] Applying Lemma \ref{Le:4} to (\ref{N:16}), we get $s=\tilde{s}$ near $O$, because the difference $s-\tilde{s}$ on the right hand side also satisfies the elliptic equation (\ref{N:15}). Applying the unique continuation for elliptic equation (see  \cite[Theorem 17.2.6, Chapter XVII]{UCP}) gives $s\equiv\tilde{s}$ on $D$.
\end{itemize}
\hfill$\Box$
\section{Appendix:  Proof of Lemma \ref{Le:1}}
\tcb{Let $m\geq 2$ be the order of the weakly singular point $O\in\Gamma$ and let $k\in \N$ be the integer specified in Lemma \ref{Le:1} where $2\leq n<m$ is the integer satisfying (\ref{n}). Recall that $g:=w|_\Gamma$ and $h:=\partial_2w|_\Gamma$ are given by (\ref{g}) and (\ref{h}), respectively. It follows from (\ref{eq:4}) and (\ref{eq:i2}) that $g_j^\pm=h_j^\pm=0$ for all $j\in\N_0$.}
\tcr{Before proving Lemma \ref{Le:1}, we need to introduce several new index sets and prepare some lemmas.}
\begin{definition} \label{d.1}
For $(i,j)\in \Upsilon_{\alpha}^a$, we denote the set of all pair sequences
${\{(i_k,j_k)\}}^d_{k=1}$ satisfying (\ref{e1}) by $\Lambda^{a}_{i,j,\alpha}$. Here, $\Upsilon_{\alpha}^a$ is given by Definition \ref{d}. Furthermore, we define the set $\Lambda^{a}_{i,j,\alpha,m}:= \big\{{\{(i_k,j_k)\}}^d_{k=1}\in \Lambda^{a}_{i,j,\alpha}: i_k<m \big\}$ and denote by
\begin{align} \label{e2}
\mathcal M_{\alpha}^a := \sum_{(i,j)\in \Upsilon_{\alpha}^{a}}\frac{C^{\alpha}_j a_{i,j}}{i!j!} \sum_{{\{(i_k,j_k)\}}^d_{k=1}\in \Lambda^{a}_{i,j,\alpha,m}}
\sum^d_{k=1}{{\left(\frac{f_{i_k}}{i_k!}\right)}^{j_k}C^{j_k}_{j-\alpha-\sum\limits_{l<k} j_l}}.
\end{align}
Here, $C^a_b:=b!/(a! (b-a)!)$ denotes the combination number with $0\leq a\leq b$.
\end{definition}

\begin{remark}\label{r1}
From the definition of $\mathcal M_{\alpha}^a$, it is easily seen that 
\begin{align*}
\mathcal M_{\alpha}^a x_1^a = \sum_{(i,j)\in \Upsilon_{\alpha}^{a}}\frac{C^{\alpha}_j a_{i,j}}{i!j!}x^i_1 \sum_{{\{(i_k,j_k)\}}^d_{k=1}\in \Lambda^{a}_{i,j,\alpha,m}}
\sum^d_{k=1}{{\left(\frac{f_{i_k}}{i_k!}x^{i_k}_1\right)}^{j_k}C^{j_k}_{j-\alpha-\sum\limits_{l<k} j_l}}.
\end{align*}
\tcr{Note that we have $i+i_k j_k=a$ for $(i,j)\in \Upsilon_{\alpha}^a$ and
${\{(i_k,j_k)\}}^d_{k=1}\in\Lambda^{a}_{i,j,\alpha,m}$ by definition of the index sets. }
\end{remark}
\tcr{Now we make use of $\mathcal M_{\alpha}^a\in \C$  to express} $g^{l}_{\pm}$ and $h^{l}_{\pm}$ for some $l$.

\begin{lemma} \label{l}
\tcr{Let the assumptions in Lemma \ref{Le:1} hold true}.
\begin{enumerate}[(a)]
\item \tcr{For $-(m+k)\le l<m$}, we have
\begin{align} \label{e8}
\frac{g^{\pm}_{m+l+k}}{(m+l+k)!} =  \sum_{\lambda \in J }2\mathcal M^{\lambda+k}_2 \frac{f^{\pm}_{l+m-\lambda}}{(l+m-\lambda)!}+ \mathcal M^{l+m+k}_1.
 \end{align}
Here, $J:=\{\lambda\in \mathbb Z: -k\le\lambda \le l\}$. If $J = \emptyset$, then we denote $\sum_{\lambda \in J }= 0$.
\tcr{For $l=m$}, we have
\begin{align} \label{e9}
\frac{{g^{\pm}_{2m+k}}}{(2m+k)!}={\sum_{m<l\le 2m+k}\frac{f^{\pm}_{l}}{l!}}2\mathcal M^{{2m-l}+k}_2+\frac{3a_{k,3}}{k!3!}\frac{{(f^{\pm}_m)}^2}{{(m!)}^2}
+\frac{{f^{\pm}_m}}{m!}2\mathcal M^{m+k}_2 + \mathcal M^{2m+k}_1.
\end{align}
\item
\begin{align}\label{e13}
\frac{h^{\pm}_{3m+k}}{(3m+k)!}&=\sum_{ m<l\le 3m+k}\frac{f^{\pm}_lg^{\pm}_{3m-l+k}}{l!(3m-l+k)!}- \sum_{m<l\le \frac{3m+k}2}\left(\frac{f^{\pm}_l}{l!}\right)^2  \mathcal M^{3m-2l+k}_{2} +\frac{{(f^{\pm}_m)}^3}{{(m!)}^3}\frac{a_{k,3}}{k!3!}\\
&+\frac{{(f^{\pm}_m)}^2}{{(m!)}^2}\mathcal M^{m+k}_2+\frac{f^{\pm}_{m}}{m!}\mathcal M^{2m+k}_1
+\mathcal M^{3m+k}_0. \notag
\end{align}
\end{enumerate}
\end{lemma}
\begin{proof}
(a)
Denote by $A^{\pm}:=\{f_j^{\pm}: j\ge m\}$, $A_1^{\pm}:=\{f_j^{\pm}: j> m\}$ \tcr{and recall (e.g., \eqref{g})}
\begin{align*}
g(x_1)=\sum_{i\ge 0, j\ge1} \frac{ja_{i,j}}{i !\,j !} x_1^i\left( \sum_{l=n}^{m-1} \frac{f_l}{j!} x_1^l  +   \sum_{l=m}^{\infty} \frac{f_l^\pm}{j!} x_1^l  \right)^{j-1} = \sum_{l\in \N_0}\frac{g^l_\pm}{\tcr{l\,!}} x_1^l \quad \mbox{for}\quad x_1\lessgtr 0.
\end{align*}
Then, by Remark \ref{r1} straightforward calculations show that, when $-(m+k)\le l<m$,
\begin{align*}
\frac{g^{\pm}_{m+l+k}}{(m+l+k)!}x_1^{m+l+k} = \bigg(\sum_{\lambda \in J} \frac{2\mathcal M^{\lambda+k}_2 f^{\pm}_{l+m-\lambda}}{(l+m-\lambda)!}+ \mathcal M^{l+m+k}_1+\eta_{l}^{\pm}\bigg)x_1^{m+l+k}, \quad x_1\lessgtr 0.
\end{align*}
Here, $\eta_l^{\pm}$ are the sum of monomials for \tcr{$f_l\in A^{\pm}$ which rely} on at least two elements from the set $A^{\pm}$. \tcr{In fact, the numbers} $\eta_l^{\pm}$  can be obtained by inserting (\ref{eq:60}) into $\sum_{(i,j)\in J_{k,l}} {ja_{i,j}}/({i !\,j !}) x_1^ix_2^j$, where $J_{k,l}:= \{(i,j): i\le l+m+k-(2m+(j-3)n),\; j\ge 3 \}$. When $l<m$, it is easy to verify that $i+jn <k+3n$ for $(i,j)\in J_{k,l}$. Then, \tcr{in view of the induction hypothesis} on $a_{i,j}$ in Lemma \ref{Le:1} (see formula (\ref{er})), it can be deduced that $a_{i,j} =0$ \tcr{for} $(i,j)\in J_{k,l}$, implying $\eta^{\pm}_l = 0$. 
\tcr{This thus proves the equality (\ref{e8})}.

\tcr{To prove \eqref{e9}, we use Remark \ref{r1} again and compare} the coefficients of $x_1^{2m+k}$ in the expansion of $g$. It then follows that
\begin{align*}
\frac{{g^{\pm}_{2m+k}}}{(2m+k)!}={\sum_{m<l\le 2m+k}\frac{f^{\pm}_{l}}{l!}}2\mathcal M^{{2m-l}+k}_2+\frac{3a_{k,3}}{k!3!}\frac{{(f^{\pm}_m)}^2}{{(m!)}^2}
+\frac{{f^{\pm}_m}}{m!}2\mathcal M^{m+k}_2 + \mathcal M^{2m+k}_1  + \zeta^{\pm} + \zeta^{\pm}_m.
\end{align*}
Here, $\zeta^{\pm}$ are the sum of some monomials \tcr{relying on} $\Pi^{p}_{j=1} f_{j}^{\pm}$ with $f_{1}^{\pm}\in A_1^{\pm}$ and $f_{2}^{\pm}, \cdots f_{p}\in  A^{\pm}$ for some $p\ge 2$; $\zeta^{\pm}_m$ are the sum of monomials only depending on $(f_m^{\pm})^p$ for some $p\ge 3$ and on the coefficients $a_{i,j}$. \tcr{Arguing analogously to} $\eta^{\pm}_l$, we deduce from the assumptions on $a_{i,j}$  that $\zeta^{\pm}=\zeta^{\pm}_m=0$. This implies that (\ref{e9}) holds.

(b) \tcr{We first recall the expansion of $h$ by}
\begin{align*}
h(x_1)=\sum_{i\ge 0, j\ge0} \frac{a_{i,j}}{i !\,j !} x_1^i\left( \sum_{l=n}^{m-1} \frac{f_l}{j!} x_1^l  +   \sum_{l=m}^{\infty} \frac{f_l^\pm}{j!} x_1^l  \right)^{j} = \sum_{l\in \N_0}\frac{h^l_\pm}{\tcr{l\,!}} x_1^l \quad \quad x_1\lessgtr 0.
\end{align*}
Using Remark \ref{r1} and equating the coefficients of $x_1^{3m+k}$ in the above expansion, we get
\begin{align} \label{e12}
\frac{h^{\pm}_{3m+k}}{(3m+k)!}  =& \sum_{m<l\le 3m+k} \frac{f_l^{\pm}}{l!} \left(\sum_{\lambda \in J_1 } \frac{2\mathcal M^{\lambda+k}_2 f^{\pm}_{3m-l-\lambda}}{(3m-l-\lambda)!}+ \mathcal M^{3m-l+k}_1\right) + \sum_{m<l\le \frac{3m+k}2}\left(\frac{f^{\pm}_l}{l!}\right)^2  \mathcal M^{3m-2l+k}_{2} \notag\\
&\tcr{+}\frac{{(f^{\pm}_m)}^3}{{(m!)}^3}\frac{a_{k,3}}{k!3!}+\frac{{(f^{\pm}_m)}^2}{{(m!)}^2}\mathcal M^{m+k}_2+\frac{f^{\pm}_{m}}{m!}\mathcal M^{2m+k}_1 +\mathcal M^{3m+k}_0 + \gamma^{\pm}+ \gamma^{\pm}_m,
\end{align}
where $J_1:=\{\lambda \in \mathbb Z: -k\le\lambda \le 2m-l,
\lambda \ne 3m-2l\}$\tcr{;}  $\gamma^{\pm}$ are the sum of monomials depending on $\Pi^{p}_{j=1} f_{j}^{\pm}$ with $f_{1}^{\pm}\in A_1^{\pm}$ and $f_{2}^{\pm}, \cdots f_{p}\in  A^{\pm}$ for some $p\ge 3$\tcr{;} $\gamma^{\pm}_m$ depends on $(f_m^{\pm})^p$ for some $p\ge 4$ and the coefficients $a_{i,j}$. \tcr{Employing similar arguments as for $\eta_l^{\pm}$ in the assertion (a), it follows from the induction hypothesis} on $a_{i,j}$ that $\gamma^{\pm} = \gamma^{\pm}_m = 0$. On the other hand, it is easily seen that $3m-l<2m$ when $l>m$. Thus, from (\ref{e8}), we have
\begin{align*}
\frac{g^{\pm}_{3m-l+k}}{(3m-l+k)!} = \sum_{-k\le\lambda \le 2m-l } \frac{2\mathcal M^{\lambda+k}_2 f^{\pm}_{3m-l-\lambda}}{(3m-l-\lambda)!}+ \mathcal M^{3m-l+k}_1.
\end{align*}
This together with (\ref{e12}) and the fact that $\gamma^{\pm} = \gamma^{\pm}_m = 0$ \tcr{proves the relation} (\ref{e13}).
\end{proof}

 To proceed we introduce two new sets of indices.
 \begin{definition} 
Given $j \in \N, j<m $, the set $\digamma_j$ is  defined by
\[
\digamma_j=\{a:\;\; \exists\; b>m,  j+m=a+b,\; a\ge n ,\; f^{+}_{b}\ne f^{-}_{b} \}
\]
and the subset $\digamma_{j,1}$ of $\digamma_{j}$ is defined by
\[
\digamma_{j,1}=\{a\in \digamma_j: \;\;\exists\; a_1\ge n, b_1>m, \ a+m=a_1+b_1, f^{+}_{b_1}\ne f^{-}_{b_1} \}.
\]
\end{definition}

\begin{lemma} \label{l1}
Under the assumptions of Lemma \ref{Le:1}, the relation $g(x_1)\equiv 0$ for all $x_1\in(-1/2,1/2)$ (see (\ref{g}) for the definition of $g$) implies that
\begin{align} \label{eq:23}
\mathcal M^{l+k}_2=0\quad \mbox{for all}\quad -k \le l<m.
\end{align}
\end{lemma}

\begin{proof}
\tcr{Assume $ -k \le l <n$}. From the definition of the index sets $\Upsilon_{\alpha}^a$ (see Definition \ref{d}), it follows for $(i,j) \in \Upsilon^{l+k}_2$ that,
\begin{align*}
&i = l+k, \;\;\; \quad \quad
\quad \quad \quad \;  \textrm{if}\; j=2,\\
&i\le l+k-(j-2)n, \;\quad  \textrm{if} \; j\ge 3.
\end{align*}
\tcr{This leads to}
\begin{align*}
i+jn \le l+2n+k < 3n+k \quad\; \textrm{if} \quad (i,j) \in \Upsilon^{l+k}_2,
\end{align*}
which together with the induction hypothesis of Lemma \ref{Le:1} yields  $a_{i,j} =0$ when $(i,j) \in \Upsilon^{l+k}_2$. By definition \ref{d.1}, this implies  
\begin{align}\label{e11}
\mathcal M^{l+k}_2=0\quad \; \textrm{for} \;  -k \le l <n.
\end{align}
\tcr{Below we only need to consider the case of $l\ge n$. In the sequel, we set $f_p = f_p^+ = f_p^-$ if $f_p^+=f_p^-$ for some $p>m$.}

{We first assume that $\digamma_{l}$ is empty. 
\tcr{Combining} (\ref{e8}) in Lemma \ref{l}, the relations in (\ref{e11}) and $g^{\pm}_{m+l+k}=0$, we arrive at
}
\begin{align*}
 2\mathcal M^{l+k}_2\frac{f_{m}^{+}}{m!}+ \sum_{n\le\lambda < l } \frac{2\mathcal M^{\lambda+k}_2 f_{l+m-\lambda}}{(l+m-\lambda)!}+ \mathcal M^{l+m+k}_1 = 2\mathcal M^{l+k}_2\frac{f_{m}^{-}}{m!} +  \sum_{n\le\lambda < l } \frac{2\mathcal M^{\lambda+k}_2 f_{l+m-\lambda}}{(l+m-\lambda)!}+ \mathcal M^{l+m+k}_1,
\end{align*}
implying that
\begin{align}\label{e7}
\mathcal M^{l+k}_2\left(\frac{f_{m}^{+}}{m!}-\frac{f_{m}^{-}}{m!}\right)=0.
\end{align}
This together with the fact that $f^{+}_m \ne f^{-}_m $ proves (\ref{eq:23}).

{Now, suppose that $\digamma_{l}\neq\emptyset$.
Again using (\ref{e8}), (\ref{e11}) and  $g^{\pm}_{m+l+k}=0$} we deduce  that

\begin{align}
\label{eq:22}
&\sum_{\lambda_0\in \digamma_{l}}\frac{2f^{\pm}_{m+l-\lambda_0}}{(m+l-\lambda_0)!}\mathcal M^{\lambda_0+k}_2
+ \frac{2f^{\pm}_{m}}{m!}\mathcal M^{l+k}_2
=\mathcal M^{m+l+k}_1-\sum_{n\le\lambda_0 < l, \lambda_0 \notin \digamma_{l} } \frac{2\mathcal M^{\lambda_0+k}_2 f_{l+m-\lambda_0}}{(l+m-\lambda_0)!}.
\end{align}
Split the set $\digamma_{l}$ into two classes $\digamma_{l,1}$ and $\digamma_{l} \backslash \digamma_{l,1}$.
For ${\color{hw}{\lambda_0}} \in \digamma_{l} \backslash \digamma_{l,1}$, similar to the derivation of (\ref{e7}), it follows from (\ref{e8}), (\ref{e11}) and $g^{\pm}_{\lambda_0+m+k}=0$  that
\begin{align*}
&\mathcal M^{\lambda_0+k}_2\left(\frac{f_{m}^{+}}{m!}-\frac{f_{m}^{-}}{m!}\right)=0,
\end{align*}
leading to
\begin{align} \label{Ne:201}
\mathcal M^{\lambda_0+k}_2=0.
\end{align}

For any $\lambda_0\in \digamma_{l,1}$, one can deduce from (\ref{e8}), (\ref{e11}) and $g^{\pm}_{\lambda_0+m+k}=0$ that
\begin{align*}
\sum_{\lambda_1\in \digamma_{\lambda_0}}\frac{2f^{\pm}_{\lambda_0+m-\lambda_1}}{(\lambda_0+m-\lambda_1)!}\mathcal M^{\lambda_1+k}_2 +
\frac{2\mathcal M^{\lambda_0+k}_2 f^{\pm}_{m}}{m!} 
 = \mathcal M^{\lambda_0+m+k}_1-\sum_{n\le\lambda_0 < \lambda_1, \lambda_1 \notin \digamma_{\lambda_0} } \frac{2\mathcal M^{\lambda_1+k}_2 f_{\lambda_0+m-\lambda_1}}{(\lambda_0+m-\lambda_1)!}.
\end{align*}

Repeating this process we can divide the set $\digamma_{\lambda_0}$ into $\digamma_{\lambda_0,1}$ and $\digamma_{\lambda_0} \backslash \digamma_{\lambda_0,1}$ to generate a new set $\digamma_{\lambda_1}$. Then we split $\digamma_{\lambda_1}$  into $\digamma_{\lambda_1, 1}$ and $\digamma_{\lambda_1} \backslash \digamma_{\lambda_1,1}$ to continue this process. After a finite number of steps we may end up this process \tcr{with the empty set} $\digamma_{\lambda_r}=\emptyset$ for some $\lambda_r \geq n$. In this way we can get the sequence
${\lambda_0>\lambda_1>\cdots \ge \lambda_r}$. For simplicity, we assume that there is only one chain $\lambda_0\rightarrow \lambda_1 \rightarrow \cdots \rightarrow \lambda_r$ and \tcr{that} $\digamma_{\lambda_r}$ is the first empty set.   The case of multiple chains can be proved similarly. Further, with the aid of (\ref{e8}), (\ref{e11}) and $g^{\pm}_{\lambda_l+m+k}=0$  with $0\le l\le r$, we have for each $l , 0 \le l \le {\color{hw}{r-1}}$ that
\begin{align} \nonumber
&\sum_{\lambda_{l+1}\in \digamma_{\lambda_l}}\frac{2\mathcal M^{\lambda_{l+1}+k}_2f^{\pm}_{\lambda_{l}+m-\lambda_{l+1}}}{(\lambda_l+m-\lambda_{l+1})!} +
\frac{2f^{\pm}_{m}}{m!} \mathcal M^{\lambda_l+k}_2\\ \label{eq:20}
 =& \mathcal M^{\lambda_l+m+k}_1-\sum_{n\le\lambda_0 < \lambda_l, \lambda_{l+1} \notin \digamma_{\lambda_l} } \frac{2\mathcal M^{\lambda_{l+1}+k}_2 f_{\lambda_{l}+m-\lambda_{l+1}}}{(\lambda_l+m-\lambda_{l+1})!}
\end{align}
and
\begin{align*}
 2\mathcal M^{\lambda_r+k}_2\frac{f_{m}^{\pm}}{m!}= - \sum_{n\le\lambda < \lambda_r } \frac{2\mathcal M^{\lambda+k}_2 f_{\lambda_r+m-\lambda}}{(\lambda_r+m-\lambda)!}- \mathcal M^{\lambda_r+m+k}_1.
\end{align*}
 The \tcr{last} identity together with the fact that $f^{+}_m \ne f^{-}_m $ implies $\mathcal M^{\lambda_r+k}_2 =0$.
This combined with (\ref{eq:20}) for $l=r-1$ gives the relation
\[
\mathcal M^{\lambda_{l-1}+k}_2=0.
\]
Repeating the same arguments, we can obtain (\ref{Ne:201}). 
\tcr{Combining}  with  (\ref{eq:22}) and (\ref{Ne:201}) gives $(\ref{eq:23})$. The proof is thus completed.
\end{proof}

Now we are in a position to \tcr{finish} the proof of  Lemma \ref{Le:1}.

\begin{proof}[Proof of  Lemma \ref{Le:1}]
Let
\begin{align*}
D_{k,l} = \mathcal M^{lm+k}_{3-l} \; \; \textrm{for} \; \; l=1,2,3.
\end{align*}
From the formula (\ref{e2}) in Definition \ref{d.1}, it is deduced that
$D_{k, l}\in \C$ depends on $a_{i,j}$ with $(i,j) \in \Upsilon^{lm+k}_{3-l}$ and only on the coefficients $f_l$ of $f(x_1)$ with $l<m$.

We first prove (\ref{eq:16}). Using (\ref{e9}) in Lemma \ref{l} and $g^{\pm}_{2m+k} = 0$, we get
\begin{align} \label{e4}
&{\sum_{m<l\le 2m+k}\frac{f^{\pm}_{l}}{l!}}2\mathcal M^{{2m-l}+k}_2+\frac{3a_{k,3}}{k!3!}\frac{{(f^{\pm}_m)}^2}{{(m!)}^2}
+\frac{{f^{\pm}_m}}{m!}2D_{k,1} + D_{k,2}= 0.
\end{align}
Note that $2m-l<m$ when $m<l<2m$.
The relation (\ref{eq:16}) then follows from
 Lemma \ref{l1} and (\ref{e4}).

Now we need to prove (\ref{eq:15}). Recalling (\ref{e13}) in Lemma \ref{l} and $h^{\pm}_{3m+k} = 0$, we have 
\begin{align*}
0=&\sum_{ m<l\le 3m+k}\frac{f^{\pm}_lg^{\pm}_{3m-l+k}}{l!(3m-l+k)!}- \sum_{m<l\le \frac{3m+k}2}\left(\frac{f^{\pm}_l}{l!}\right)^2  \mathcal M^{3m-2l+k}_{2}  +\frac{{(f^{\pm}_m)}^3}{{(m!)}^3}\frac{a_{k,3}}{k!3!}\\
&+\frac{{(f^{\pm}_m)}^2}{{(m!)}^2}D_{k,1}+\frac{f^{\pm}_{m}}{m!}D_{k,2}
+D_{k,3}. \notag
\end{align*}
{Note that $g^{\pm}_{3m-l+k}$ denotes the $(3m-l+k)$th coefficient of $g(x_1)$ defined by (\ref{g}).} Further, it is easily seen that $3m-2l<m$ when $l>m$. Thus, using the results of Lemma \ref{l1} together with $g^{\pm}_{3m-l+k}=0$,  it follows that (\ref{eq:15}) holds.
The proof is thus complete.
 \end{proof}

\section{Acknowledgments}
The work of L. Li and J. Yang are partially supported by the National Science Foundation of
China (11961141007,61520106004) and Microsoft Research of Asia.
The work of G. Hu is supported by the National Natural Science Foundation of China
 (No. 12071236) and the Fundamental Research Funds for Central Universities in China (No. 63213025).

\end{document}